\documentclass[12pt]{amsart}
\usepackage[latin1]{inputenc}
\usepackage{indentfirst}
\usepackage{amsfonts}
\usepackage{graphicx}
\usepackage{amsmath}
\usepackage{amssymb}
\usepackage{latexsym}
\usepackage{amscd}
\usepackage{multirow}
\usepackage{xcolor}
\usepackage{float}

\newcommand{\F}{\mathbb {F}}

\newtheorem{theorem}{Theorem}[section]
\newtheorem{definition}[theorem]{Definition}
\newtheorem{lemma}[theorem]{Lemma}

\newtheorem{proposition}[theorem]{Proposition}

\def \ord {\rm{ord}\,}

\begin{document}

\title[Primitive $2$-normal elements]{Existence of primitive $2$-normal elements
in finite fields}

\author{Josimar J. R. Aguirre and Victor G.L. Neumann}

\maketitle


\vspace{8ex}

\begin{abstract}
An element $\alpha \in \mathbb{F}_{q^n}$ is normal over $\F_q$ if $\mathcal{B}=\{\alpha, \alpha^q, \alpha^{q^2}, \cdots, \alpha^{q^{n-1}}\}$ forms a basis of $\F_{q^n}$ as a vector space over $\F_q$. It is well known that $\alpha \in \F_{q^n}$ is normal over $\F_q$ if and only if $g_{\alpha}(x)=\alpha x^{n-1}+\alpha^q x^{n-2}+ \cdots + \alpha^{q^{n-2}}x+\alpha^{q^{n-1}}$ and $x^n-1$ are relatively prime over $\F_{q^n}$, that is, the degree of their greatest common divisor in $\F_{q^n}[x]$ is $0$.
Using this equivalence, the notion of $k$-normal elements was introduced in Huczynska et al. ($2013$): an element $\alpha \in \F_{q^n}$ is $k$-normal over $\F_q$ if the greatest common divisor of the polynomials $g_{\alpha}[x]$ and $x^n-1$ in $\F_{q^n}[x]$ has degree $k$; so an element
which is normal in the usual sense is $0$-normal.

Huczynska et al. made the question about the pairs $(n,k)$ for which there exist primitive $k$-normal elements in $\F_{q^n}$ over $\F_q$ and they got a partial result for the case $k=1$, and later Reis and Thomson ($2018$) completed this case. The Primitive Normal Basis Theorem solves the case $k=0$. In this paper, we solve completely the case $k=2$ using estimates for Gauss sum and the use of the computer, we also obtain a new condition for the existence of $k$-normal elements in $\F_{q^n}$.

\end{abstract}

\section{Introduction}

Let $\mathbb{F}_{q^n}$ be a finite field with $q^n$ elements, where $q$ 
is a prime power and $n$ is a positive integer.
An element $\alpha \in \mathbb{F}_{q^n}^*$ is primitive if $\alpha$ generates
the cyclic multiplicative group $\mathbb{F}_{q^n}^*$ ($\alpha$ has multiplicative
order $q^n-1$). Also, $\alpha \in \mathbb{F}_{q^n}$ is normal over $\mathbb{F}_{q}$ if the set
$B_{\alpha}=\{ \alpha^{q^i} \mid 0 \leq i \leq n-1 \}$ spans $\mathbb{F}_{q^n}$
as a $\mathbb{F}_q$-vector space, in this case we say that $B_{\alpha}$ is
a normal basis.
Normal basis are frequently used in cryptography and computer algebra systems due
to the efficiency of exponentiation. Primitive
elements are constantly used in cryptographic applications such as discrete
logarithm problem and pseudorandom number generators \cite{meletiou}. If we put these two properties together, we obtain a \textit{primitive normal element}. We can study the multiplicative structure of $\mathbb{F}_{q^n}$ and at the same time see $\mathbb{F}_{q^n}$ as a vector space over $\mathbb{F}_q$. The \textit{Primitive Normal Basis Theorem} states that for any extension field $\mathbb{F}_{q^n}$ of $\mathbb{F}_q$, there exists a basis composed by primitive normal elements; this result was first proved by Lenstra and Schoof \cite{lenstra} using a combination of character sums, sieving results and a computer search.

One can prove that an element $\alpha \in \mathbb{F}_{q^n}$ is normal if and only if the polynomial $g_{\alpha}(x)=\alpha x^{n-1}+ \alpha^q x^{n-2} + \ldots + \alpha^{q^{n-2}} x + \alpha^{q^{n-1}}$ and $x^n-1$ are relatively prime over $\mathbb{F}_{q^n}$
\cite[Theorem 2.39]{LN}. With this as motivation, Huczynska et al. \cite{knormal} introduced the concept of $k$-normal elements, as an extension of the usual normal elements:

\begin{definition}
Let $\alpha \in \mathbb{F}_{q^n}$ and let $g_{\alpha}(x) = \sum_{i=0}^{n-1} \alpha^{q^i}x^{n-1-i} \in \mathbb{F}_{q^n}[x]$. If $\gcd(x^n-1,g_{\alpha}(x))$ over $\mathbb{F}_{q^n}$ has degree $k$ (where $0 \leq k \leq n-1$), then $\alpha$ is a $k$-normal element of $\mathbb{F}_{q^n}$ over $\mathbb{F}_q$. \footnote{We use this definition to find primitive $2$-normal elements for specific values of $(q,n)$, using Sagemath (cf. \cite{SAGE}) program. Also, throughout this paper, when we talk about a $k$-normal element $\alpha \in \mathbb{F}_{q^n}$, it will be over $\mathbb{F}_q$.}
\end{definition}


The $k$-normal elements can be used to reduce the multiplication process in finite fields, see \cite{negre}. From the above definition,
elements which are normal in the usual sense are $0$-normal and from the Primitive Normal Basis Theorem, we know that they always exist.
There are several criteria in the literature for the existence of $k$-normal elements 
(see for example \cite{lucas}, \cite{sozaya}, \cite{zhang}).
In \cite{knormal} the authors worked out the case $k=1$, and partially established a Primitive $1$-normal element Theorem. Reis and Thompson completed the case $k=1$ in \cite{lucas1}.

A question which naturally arises is: for which values of $k$ one has a Primitive $k$-normal element Theorem (see \cite[Problem 6.3]{knormal})?
On this line, in \cite{lucas}, Reis obtained a sufficient condition for the existence of primitive $k$-normal elements, and he proved that \textit{given $\epsilon>0$, for $q$ sufficiently large, 
there exist primitive $k$-normal elements for $k \in [0,(\frac{1}{2}-\epsilon)n]$, whenever $k$-normal elements actually exist in $\mathbb{F}_{q^n}$}. Since this is an asymptotic result for $q$, it is not possible to conclude the result for specific values, but from
the condition that he obtained
it is possible to generalize and study particular cases of $k$.


Since the cases $k=0$ and $k=1$ are completely finished, in this paper we study the case $k=2$ as follows: in Section \ref{sectionpre}, we provide background material that is used along the
paper. In Section \ref{sectiongen}, we present two general conditions for the existence of primitive $k$-normal elements in $\F_{q^n}$ over $\F_q$, as well as some weaker conditions for some particular cases.
In Section \ref{sectionalmost}, we apply the results from previous sections
to prove
all cases for $n \geq 8$ and also for $q \leq 19$.
In Section \ref{section=567}, we study the cases $n=5,6,7$ developing new ideas based on the factorization of certain divisors of $q^n-1$. Finally, in Section \ref{section=4}
we study the remaining case $n=4$,
where we prove  that there exist primitive $2$-normal elements
in $\F_{q^4}$ if and only if $q \equiv 1 \pmod 4$. For this last case
we develop  Gauss sum which is different from the ones used in the previous cases.

Our results can be summarized in the following theorem.

\begin{theorem}[\bf{The Primitive $2$-Normal Theorem}]\label{Final}
Let $q$ be a prime power and $n$ be a natural number.
There exists a primitive $2$-normal element in $\F_{q^n}$ if and only if
$n \geq 5$ and
$\gcd(q^3-q,n) \neq 1$ or $n=4$ and $q \equiv 1 \pmod 4$.
\end{theorem}

In Appendix A we show the SageMath procedures that we used in the paper
and in Appendix B we present tables with primitive $2$-normal
elements for specific cases. 

\section{Preliminaries}\label{sectionpre}

In this section, we present some definitions and results that will be useful in the rest of
this paper. We start with the following definitions.

\begin{definition}
\begin{enumerate}
\item[(a)] Let $f(x)$ be a monic polynomial with coefficients in $\mathbb{F}_{q}$. The Euler Totient Function for polynomials over $\mathbb{F}_q$ is given by
$$
\Phi_q(f)= \left| \left( \dfrac{\mathbb{F}_q[x]}{\langle f \rangle} \right)^{*} \right|,
$$
where $\langle f \rangle$ is the ideal generated by $f(x)$ in $\mathbb{F}_q[x]$.
\item[(b)] If $t$ is a positive integer (or a monic polynomial over $\mathbb{F}_q$), $W(t)$ denotes the number of square-free (monic) divisors of $t$.
\item[(c)] If $f(x)$ is a monic polynomial with coefficients in $\mathbb{F}_q$, the Polynomial Möbius Function $\mu_q$ is given by $\mu_q(f)=0$ if $f$ is not square-free and $\mu_q(f)=(-1)^r$ if $f$ is a product of $r$ distinct irreducible factors over $\mathbb{F}_q$.
\end{enumerate}
\end{definition}

We have an interesting formula for the number of $k$-normal elements over
finite fields:

\begin{theorem}\label{counting-knormal}
(\cite{knormal}, Theorem 3.5) The number of $k$-normal elements of $\mathbb{F}_{q^n}$ over $\mathbb{F}_q$ is given by
$$
	\sum_{\substack{h| x^n-1 \\ \deg(h)=n-k}} \Phi_q(h),
$$
where the divisors are monic and the polynomial division is over $\mathbb{F}_q$.
\end{theorem}

\subsection{Linearized polynomials and the $\mathbb{F}_q$-order.} Here we present some definitions and basic results on linearized polynomials over finite fields that are frequently used in this paper. A useful feature of these polynomials is the structure of the set of roots that facilitates the determination of the roots, see \cite[Section 3.4]{LN}.

\begin{definition}
Let $f \in \mathbb{F}_q[x]$ with $f(x)=\sum_{i=0}^r a_ix^i$.
\begin{enumerate}
\item[(a)] The polynomial $L_f(x):=\sum_{i=0}^r a_ix^{q^i}$ is the 
linearized $q$-associate of $f$.
\item[(b)] For $\alpha \in \mathbb{F}_{q^n}$, we set 
$L_f(\alpha)= \sum_{i=0}^r a_i \alpha^{q^i}$.
\end{enumerate}
\end{definition}

The polynomial $L_f$
induces a linear transformation of $\mathbb{F}_{q^n}$ over $\F_q$ that also has additional properties:

\begin{lemma}\cite[Lemma 3.59]{LN}
Let $f,g \in \mathbb{F}_q[x]$. The following hold:
\begin{enumerate}
\item[(a)] $L_f(x)+L_g(x)=L_{f+g}(x)$;
\item[(b)] $L_{fg}(x)=L_f \left( L_g(x) \right) = L_g \left( L_f(x) \right)$.
\end{enumerate}
\end{lemma}

\begin{lemma}\label{kernel}
Let $f,g \in \F_q[x]$ such that $fg=x^n-1$. For
every $\alpha \in \F_{q^n}$, we have that
$L_g(\alpha)=0$ if and only if $\alpha=L_f(\beta)$ for some
$\beta \in \F_{q^n}$.
\end{lemma}
\begin{proof}
Observe that $fg=x^n-1$ implies 
$L_g \circ L_f =L_f \circ L_g = L_{x^n-1} = 0$,
so $\text{Im}\, L_f \subset \text{Ker}\, L_g$ and
$\text{Im}\, L_g \subset \text{Ker}\, L_f$.
On the other hand, $L_g$ has degree $q^{\deg g}$ and $\text{Ker}\, L_g$ has
at most dimension $\deg g$. Conversely,
we have that $\text{Im}\, L_g$ has at most dimension $\deg f=n-\deg g$. So,
we get that $\text{Ker}\, L_g$ has dimension exactly
$\deg g$
and $\text{Im}\, L_f = \text{Ker}\, L_g$, since
$n= \dim_{\F_q} \text{Im}\, L_g + \dim_{\F_q} \text{Ker}\, L_g
\leq (n - \deg g) + \deg g$.
\end{proof}

Let $D \in \F_q[x]$ be a monic polynomial.
We say that an element $\alpha \in \F_{q^n}$ has $\F_q$-order $D$
if $D$ is the lowest degree monic polynomial such that $L_D(\alpha)=0$.
It is known that the $\F_q$-order of an element $\alpha \in \F_{q^n}$
divides $x^n-1$ and we also have the following equivalences.

\begin{theorem}\label{equiv-knormal}
(\cite{knormal}, Theorem 3.2) Let $\alpha \in \mathbb{F}_{q^n}$. The following three properties are equivalent:
\begin{enumerate}
\item[(i)] $\alpha$ is $k$-normal over $\F_q$.
\item[(ii)] Let $V_{\alpha}$ be the $\F_q$-vector space generated by $\{ \alpha, \alpha^q, \ldots, \alpha^{q^{n-1}} \}$, then $\dim V_{\alpha}$ is $n-k$.
\item[(iii)] $\alpha$ has $\mathbb{F}_q$-order of degree $n-k$.
\end{enumerate} 
\end{theorem}

\subsection{Freeness and Characters.} We present the concept of \textit{freeness}, introduced in Carlitz \cite{carlitz} and Davenport \cite{davenport1},
and refined in Lenstra and Schoof (see \cite{lenstra}). This 
concept is useful in the construction of certain characteristic functions over finite fields.

\begin{definition}
\begin{enumerate}
\item[(a)] Let $m \, | \, (q^n-1)$, we say that $\alpha \in \mathbb{F}_{q^n}^*$ is $m$-free
if, for every $d \, | \, m$
and $\beta \in \F_{q^n}$, $\alpha= \beta^d$
implies that $d = 1$.
\item[(b)] Let $M \, | \, (x^n-1)$, we say that $\alpha \in \mathbb{F}_{q^n}$ is $M$-free
if, for every $h \, | \, M$
and $\beta \in \F_{q^n}$, $\alpha= L_h(\beta)$
implies that $h = 1$.
\end{enumerate}
\end{definition}

It is well known that an element $ \alpha \in \mathbb{F}_{q^n}^*$ is primitive if and only if $\alpha$ is $(q^n-1)$-free and $\alpha \in \mathbb{F}_{q^n}$ is normal if and only if $\alpha$ is $(x^n-1)$-free.

Also, from the definition, we have that if $\alpha$ is $m$-free then $\alpha$ is $e$-free, for any $e \, | \, m$ (analogous result for polynomial).











Following the notation in \cite{CH}, we can characterize the freeness of an element. For the multiplicative part: a multiplicative character $\eta$ 
of $\F_{q^n}^*$ is a
group homomorphism of $\F_{q^n}^*$ to $\mathbb{C}^*$,
whose order is the least positive integer $d$ such that
$\eta (\alpha)^d=1$ for any $\alpha \in \F_{q^n}^*$.
Let $m$ be a divisor of $q^n-1$ and let
$\int \limits_{d|m} \eta_d$ denote the sum
$ \sum_{d|m} \frac{\mu(d)}{\varphi(d)} \sum_{(d)} \eta_d$,
where $\eta_d$ is a multiplicative character of $\F_{q^n}$, and the sum
$ \sum_{(d)} \eta_d$ runs over all multiplicative characters of order $d$.
It is known that there exist $\varphi(d)$ of those characters.








For the additive part: if $p$ is the characteristic of $\F_q$, 
for $\alpha \in \F_{q^n}$,
let $\chi_\alpha : \F_{q^n} \longrightarrow \mathbb{C}$
be the additive character defined by
$$
\chi_\alpha (\beta) = e^{\frac{2\pi i}{p} \mathop{\rm{Tr}}_{q^n/p}(\alpha \beta)},
\quad \beta \in \F_{q^n},
$$
where $\mathop{\rm{Tr}}_{q^n/p}$ is the trace function of $\F_{q^n}$ over $\F_p$.
It is well known that any additive character  of $\F_{q^n}$ is of this form.
We say that the additive character $\chi_\alpha$ has $\F_q$-order $D$
if $\alpha$ has $\F_q$-order $D$.
We use the notation $\int \limits_{D|T} \chi_{\delta_D}$ to represent
$ \sum_{D|T} \frac{\mu_q(D)}{\Phi_q(D)} \sum_{(\delta_D)} \chi_{\delta_D}$ where $\chi_{\delta_D}$ runs through all characters of $\F_q$-order $D$. It is known that there exist $\Phi_q(D)$ of those characters.





For each divisor $m$ of $q^n-1$ and each monic divisor  $T \in \mathbb{F}_q[x]$ of $x^n-1$, set $\theta(m)= \dfrac{\varphi(m)}{m}$ and $ \Theta(T)= \frac{\Phi_q(T)}{q^{\deg(T)}}$.

\begin{proposition}\label{caracteristica-mlivre}
Let $m$ be a divisor of $q^n-1$ and $T \in \F_q[x]$ be a monic divisor of $x^n-1$.
For any $\alpha \in \F_{q^n}$ we have

\begin{enumerate}
\item[(i)]  $$w_m(\alpha) = 
\theta(m) \int_{d|m} \eta_d(\alpha)
=\left\{
\begin{array}{ll}
1, \quad & \text{if } \alpha \text{ is } m\text{-free}, \\
0,            & \text{otherwise.}
\end{array}
\right.
$$

\item[(ii)] $$
\Omega_T(\alpha)=
\Theta(T) \int_{D|M} \chi_D(\alpha)
=\left\{
\begin{array}{ll}
1, \quad & \text{if } \alpha \text{ is } T\text{-free}, \\
0,            & \text{otherwise.}
\end{array}
\right.
$$
\end{enumerate}
\end{proposition}
\begin{proof}

See \cite[section 5.2]{knormal} or \cite[Theorem 2.15]{lucas}. 
Extending the multiplicative characters $\eta$ to $0 \in \F_{q^n}$ by setting
$\eta(0)=0$, we can easily see that
$w_m(0)=0$.
\end{proof}

\subsection{Estimates.} To finish this section, we present some estimates that are used along the next sections.

We will need the following result, which is modeled after
\cite[Lemma 3.3]{CH} and
\cite[Lemma 4.1]{Kapetanakis-Reis} and, like these results, is proved using the 
multiplicativity of the function $W(\cdot)$ and the fact that if a positive 
integer $M$ has $l$ distinct prime divisors then $W(M) = 2^l$.

\begin{lemma}\label{cota-t}
	Let $M$ be a positive integer and $t$ be a positive real number.
	Then
	$W(M)
	\leq A_t \cdot M^{\frac{1}{t}}$,
	where
	$$
	A_t=\prod_{\substack{\wp^{\alpha_{\wp}} < 2^t \\ \wp \text{ is prime}
			\\ \wp^{\alpha_{\wp}} \mid M}}
	\frac{2}{\sqrt[t]{\wp^{\alpha_{\wp}}}},
	$$
and for any prime $\wp$, $\alpha_{\wp}$ is defined as the largest positive integer such that
$\wp^{\alpha_{\wp}} < 2^t$ and $\wp^{\alpha_{\wp}} \mid M$.
\end{lemma}
\begin{proof}
Let $M=p_1^{\alpha_1}\cdots p_l^{\alpha_l}$, so that $W(M)=2^{l}$. 
If $\wp$ is a prime such that $\wp>2^t$ then $\frac{2}{\sqrt[t]{\wp}}<1$. 
Let $\beta_i\leq \alpha_i$ be the greatest integer such that $p_i^{\beta_i} \leq 2^t$,
thus
	$$\frac{W(M)}{M^{\frac{1}{t}}}
	=\frac{2^{l}}{\sqrt[t]{p_1^{\alpha_1}}\cdots\sqrt[t]{p_l^{\alpha_l}}}
	\leq\prod_{i=1}^{l} \frac{2}{\sqrt[t]{p_i^{\beta_i}}}
	\leq
\prod_{\substack{\wp^{\alpha_{\wp}} < 2^t \\ \wp \text{ is prime}
		\\ \wp^{\alpha_{\wp}} \mid M}}
\frac{2}{\sqrt[t]{\wp^{\alpha_{\wp}}}}
=
A_t.
	$$
	The result follows immediately.
\end{proof}

Now, we present some estimates involving sum of characters:

\begin{lemma}\cite[Theorem 5.41]{LN}\label{lema cota 1}
Let $\eta$ be a multiplicative character of $\mathbb{F}_{q^n}$ of order $r>1$ and $f \in \F_{q^n}[x]$ be a monic polynomial of positive degree such that $f$ is not of the form $g(x)^r$ for some $g \in \mathbb{F}_{q^n}[x]$ with degree at least 1. Let $e$ be the number of distinct roots of $f$ in its splitting field over $\mathbb{F}_{q^n}$. For every $a \in \mathbb{F}_{q^n}$,
$$
\left| \sum_{\alpha \in \mathbb{F}_{q^n}} \eta(a f(\alpha)) \right| \leq (e-1)q^{n/2}.
$$
\end{lemma}

\begin{lemma}\cite[Theorem 2G]{schmidt}\label{lema cota 2}
Let $\eta$ be a multiplicative character of $\mathbb{F}_{q^n}$ of order $d \neq 1$ and
$\chi$ be a non-trivial additive character of $\mathbb{F}_{q^n}$.
If $F,G \in \mathbb{F}_{q^n}[x]$ are such that $F$ has exactly $m_1$ roots and
$\deg(G)=m_2$ with $gcd(d, deg(F))=gcd(m_2,q)=1$, then
$$
\left|
\sum_{\alpha  \in \mathbb{F}_{q^n}} \eta(F(\alpha)) \chi(G(\alpha))
 \right| \leq (m_1+m_2-1)q^{n/2}.
$$
\end{lemma}
%
%
%
%
%

\begin{lemma}[\cite{katz}, Theorem 1]\label{cotaenFq}
Let $F$ be a finite field,
let $n \geq 1$ be an integer and let $E$ be an extension field of $F$ of degree $n$.
Let $\chi$ be any nontrivial complex-valued multiplicative character of $E^{\times}$ (extended by zero to all of E ), and $x$ in $E$ any element that generates $E$
over $F$. Then
$$
\left| \sum \limits_{t \in F} \chi(t-x) \right| \leq (n-1) \sqrt{\# (F)}
$$
\end{lemma}

\section{General results}\label{sectiongen}

In \cite{lucas}, Reis gives a method to construct $k$-normal elements: \textit{let $\beta \in \F_{q^n}$ be a normal element and $f \in \F_q[x]$ be a divisor of $x^n-1$ of degree $k$, then $\alpha= L_f(\beta)$ is $k$-normal} (see \cite[Lemma 3.1]{lucas}).
From Theorem \ref{counting-knormal}, we also know that there exists a $k$-normal element in $\F_{q^n}$ if and only if $x^n-1$ has a divisor of degree $n-k$ (or, equivalently, a divisor of degree $k$). 
So, if $x^n-1$ has a divisor of degree $k$ and
\begin{equation}\label{eq.cond.knormal}
q^{\frac{n}{2}-k} \geq W(q^n-1) W(x^n-1),
\end{equation}
then there exists a primitive $k$-normal element in $\F_{q^n}$ (see \cite[Theorem 3.3]{lucas}).

When $k=2$, it is easy to prove that the existence of primitive $2$-normal elements is only possible for $n \geq 4$ (see Theorem \ref{equiv-knormal}). Note that we cannot use condition \ref{eq.cond.knormal}
for the case $n=4$ because the exponent on the left side is equal to zero, so we need a different approach in that case. We will discuss this case in Section \ref{section=4}. First, we are going to use the ideas of Huczynska \cite{knormal} and Reis \cite{lucas}
to get a more general result than condition \ref{eq.cond.knormal}.

From Theorem \ref{counting-knormal} we know that the existence of a factor of degree $k$ of $x ^ n-1$ is a necessary and sufficient condition for the existence of $k$-normal elements.
Thus,
the following result is very important to know the number of these factors.

\begin{lemma}\label{divisores de x^n-1} Let $q$ be a prime power and 
let $n$ be a positive integer prime to $q$.
Let $I_n(r)$ be the number of irreducible monic factors of $x^n-1$ with degree $r$ over $\F_q$. We have
$$
I_n(r)= \dfrac{1}{r} \sum_{d|r} t_d \cdot \mu \left( \frac{r}{d} \right),
$$
where $t_d=\gcd(q^d-1,n)$.
\end{lemma}

\begin{proof}

Let $\alpha$ be a primitive element in $\F_{q^r}^*$. 
For $0 \leq s \leq q^r-1$, we have
$\left( \alpha^s \right)^n=1$ if and only if $q^r-1$ divides $sn$. Since $\frac{q^r-1}{t_r}$ and $\frac{n}{t_r}$ are coprimes,
we get that
$\left(\alpha^s \right)^n=1$ if and only if $\frac{q^r-1}{t_r}$ divides $s$. Therefore, there exist $t_r$ possibilities for $s$, which implies that 
the number of elements $\alpha$ in $\F_{q^r}^*$ with $\alpha^n=1$ is 
$t_r$.

Observe that for each irreducible polynomial 
of degree $r$
defined over $\F_{q}$ which divides $x^n -1$, there are $r$ elements $\alpha$ in $\F_{q^r}$ such that $\alpha \notin \F_{q^d}$ for any $d<r$, with $\alpha^n=1$. So, from the 
definition of $I_n(r)$, the number of elements $\alpha$ in $\F_{q^r}$
which are not in $\F_{q^d}$ for $d<r$, with $\alpha^n=1$ is $r \cdot I_n(r)$.
Note that if $\alpha \in \F_{q^r} \cap \F_{q^d}$ and $d<r$, then $d \mid r$.
We conclude by using the inclusion-exclusion principle.
\end{proof}

\begin{lemma}\label{factor2}
Let $q$ be a prime power and let $n$ be a positive integer.
There exists a $2$-normal element in $\mathbb{F}_{q^n}$ over $\mathbb{F}_q$ if and only if $\gcd (q^3-q,n) \neq 1$.
\end{lemma}

\begin{proof}
The result follows directly from Theorem \ref{counting-knormal}
and Lemma \ref{divisores de x^n-1}.
%
%
%
%
\end{proof}


For the purpose of proving \textit{The Primitive Normal Basis Theorem} without computational calculations, 
in \cite{CH}, the authors defined, for $m|(q^n-1)$ and $g|(x^n-1)$, the number $N(m,g)$ of non-zero elements of $\F_{q^n}$ that are both $m$-free and $g$-free. So they needed to prove that $N(q^n-1,x^n-1)$ is positive. Similarly, we define:

\begin{definition}\label{def-NfTm}
Let $f,T \in \F_q[x]$ be divisors of $x^n-1$ such that $\deg f = k$
and let $m \in \mathbb{N}$ be a divisor
of $q^n-1$.
We denote by $N_f(T,m)$ the number of $T$-free elements $\alpha \in \mathbb{F}_{q^n}$ such that $L_f(\alpha)$ is $m$-free.	
\end{definition}


The following theorem generalizes \cite[Theorem 3.3]{lucas}
using Definition \ref{def-NfTm}.



\begin{theorem}\label{condicao em N-f}
Let $f,T \in \F_q[x]$ be divisors of $x^n-1$
such that $\deg f = k$
and let $m \in \mathbb{N}$ be a divisor
of $q^n-1$.
We have
$N_f(T,m)  >  \theta(m) \Theta(T) \big( q^n - q^{n/2+k}W(m)W(T) \big)$.
In particular,
if $q^{n/2-k} \geq W(m)W(T)$ then $N_f(T,m)>0$.
\end{theorem}
\begin{proof}
We have that 
\begin{eqnarray}
N_f(T,m) & = & \sum_{\alpha \in \mathbb{F}_{q^n}} \Omega_T(\alpha) \cdot w_m(L_f(\alpha)) 
\nonumber
\\
& = & \theta(m) \Theta(T) \sum_{\alpha \in \mathbb{F}_{q^n}} \displaystyle \int_{d|m} \displaystyle \int_{D|T} \eta_d(L_f(\alpha)) \chi_{\delta_D}(\alpha) .
\nonumber
\end{eqnarray}
If we denote the Gauss sum $S_f(\eta_d,\chi_{\delta_D})= \sum_{\alpha \in \mathbb{F}_{q^n}} 
\eta_d(L_f(\alpha)) \chi_{\delta_D}(\alpha)$, we can write
$$
\dfrac{N_f(T,m)}{\theta(m) \Theta(T)} = S_0+S_1+S_2+S_3,
$$
where $S_0=S_f(\eta_1,\chi_0)$, $S_1=\displaystyle\int \limits_{\substack{D|T\\ D \neq 1}} S_f(\eta_1, \chi_{\delta_D})$, $S_2=\displaystyle\int \limits_{\substack{d|m\\ d \neq 1}} S_f(\eta_d, \chi_{0})$ and 
$$
S_3=\displaystyle\int \limits_{\substack{D|T\\ D \neq 1}} \displaystyle\int \limits_{\substack{d|m\\ d \neq 1}} S_f(\eta_d, \chi_{\delta_D}) .
$$
We observe that
$$
S_0  =  \sum_{\alpha \in \mathbb{F}_{q^n}}  \eta_1(L_f(\alpha)) \chi_{0}(\alpha)
  =  \sum_{\alpha \in \mathbb{F}_{q^n} \backslash \text{Ker}\, L_f}  1
= q^n - q^k 
$$
and
$$
S_1  = 
\sum_{\alpha \in \mathbb{F}_{q^n} \backslash \text{Ker}\, L_f} 
 \sum_{\substack{D|T\\ D \neq 1}} \frac{\mu_q(D)}{\Phi_q(D)} \sum_{(\delta_D)}
  \chi_{\delta_D}(\alpha)
=
-
\sum_{\alpha \in \text{Ker}\, L_f} 
\sum_{\substack{D|T\\ D \neq 1}} \frac{\mu_q(D)}{\Phi_q(D)} \sum_{(\delta_D)}
\chi_{\delta_D}(\alpha),
$$
which implies that $|S_1| \leq q^k \left( W(T) - 1\right)$, since
$|\chi_{\delta_D}(\alpha)| \leq 1$.

Now we would like
good estimates of the sums $S_2$ and $S_3$. We have $f(x)=\sum_{i=0}^k a_i x_i$. One can see that the formal derivative of the $q-$associate of $f$ is $a_0 \neq 0$ (since $f$ divides $x^n-1$, $f$ is not divisible by $x$), hence $L_f$ does not have repeated roots and is not of the form 
$G(x)^r$
for any
$G(x) \in \mathbb{F}_{q^n}[x]$ and $r>1$. Therefore, by Lemma \ref{lema cota 1}, we have, for each divisor $d \neq 1$ of $q^n-1$:
$$
|S_f(\eta_d,\chi_0)| = \left| \sum_{\alpha\in \mathbb{F}_{q^n}} \eta_d(L_f(\alpha)) \right| \leq (q^k-1)q^{n/2} < q^{n/2+k}.
$$
From \ref{lema cota 2}, we conclude that, for each divisor $D \neq 1$ of $x^n-1$ and each divisor $d \neq 1$ of $q^n-1$,
$$
\left| S_f(\eta_d, \chi_{\delta_D}) \right| = \left| \sum_{\alpha \in \mathbb{F}_{q^n}} \eta_d(L_f(\alpha))\chi_{\delta_D}(\alpha)\right| \leq (q^k+1-1)q^{n/2}=q^{n/2+k}.
$$
Combining the previous bounds, we have the following inequality:
\begin{align}
N_f(T,m) & \geq  
\theta(m) \Theta(T)  
\left(
S_0 - |S_1| - |S_2| - |S_3|
\right)
\nonumber \\
& \geq 
\theta(m) \Theta(T) 
\left[ q^n - q^k  -  q^k (W(T)-1) -  q^{n/2+k}  (W(m)-1)W(T) \right] 
\nonumber \\
& >  \theta(m) \Theta(T) \big( q^n - q^{n/2+k}W(m)W(T) \big).
\nonumber
\end{align}
Therefore, if $W(m)W(T) \leq q^{n/2-k}$ we have $N_f(T,m)>0$.
\end{proof}

To use Theorem \ref{condicao em N-f}, we need to have some knowledge about the factorization
of $m$ and $T$. Knowing some factors of these values, one can use the next proposition, which helps to decrease the estimates of the function $W$ by adding an offset factor. Before that, we present a
result which will be needed in what follows.

For any natural number $m$, $rad(m)$ denotes the largest
square-free factor of $m$ and for any
polynomial $T \in \F_q[x]$, $rad(T)$ denotes the square-free factor of $T$
of largest degree over $\F_q$.

The sieving technique from the next two results follows the ideas of \cite{CH}.

\begin{lemma}\label{divisors}
Let $f,T \in \F_q[x]$ be divisors of $x^n-1$
such that $\deg f = k$
and let $m \in \mathbb{N}$ be a divisor
of $q^n-1$.
Let $Q_1,\ldots , Q_s$ be irreducible polynomials and
let $p_1,\ldots , p_r$ be prime numbers such that
$rad(x^n-1)=rad(T)\cdot Q_1 \cdot Q_2 \cdots Q_s$ and
$rad(q^n-1)=rad(m) \cdot p_1 \cdot p_2 \cdots p_r$. We have that
\begin{equation}\label{sieve}
N_f(x^n-1,q^n-1) \geq \sum_{i=1}^r N_f(T,mp_i) + \sum_{j=1}^s N_f(T \cdot Q_j,m)-(r+s-1)N_f(T,m).
\end{equation}
\end{lemma}

\begin{proof}
The left side of \eqref{sieve} counts every $\alpha \in \F_{q^n}$ for which $\alpha$ is normal and $L_f(\alpha)$ is primitive.
Observe that if $\alpha$ is normal and $L_f(\alpha)$ is primitive then
$\alpha$ is  $T \cdot Q_j$-free and $T$-free; and $L_f(\alpha)$
is $m p_i$-free and $m$-free, so $\alpha$ is  counted $r+s - (r+s-1)=1$
times on the right side of \eqref{sieve}.
For any other $\alpha \in \F_{q^n}$, we have that either $\alpha$ 
is not $T \cdot Q_j$-free for some $j \in \{ 1,\ldots , s\}$
or $L_f(\alpha)$ is not $m p_i$-free for some
$i \in \{ 1,\ldots , r\}$, so the right side of \eqref{sieve} is at most zero.
\end{proof}

\begin{proposition}\label{lema3.5}
Let $f,T \in \F_q[x]$ be divisors of $x^n-1$
such that $\deg f = k$
and let $m \in \mathbb{N}$ be a divisor
of $q^n-1$.
Let $Q_1,\ldots , Q_s$ be irreducible polynomials and
let $p_1,\ldots , p_r$ be prime numbers such that
$rad(x^n-1)=rad(T)\cdot Q_1 \cdot Q_2 \cdots Q_s$ and
$rad(q^n-1)=rad(m) \cdot p_1 \cdot p_2 \cdots p_r$.
Suppose that 
$\delta=1-\sum_{i=1}^{r}\frac{1}{p_i}-\sum_{j=1}^{s}\frac{1}{q^{\deg Q_j}}>0$
and let  $\Delta=\frac{r+s-1}{\delta}+2$. If
	$q^{\frac{n}{2}-k} \geq W(m) W(T)  \Delta$,
	then $N_f(x^n-1,q^n-1)>0$.
\end{proposition}
\begin{proof}
From equation \ref{sieve}, we have that
$$N_f(x^n-1,q^n-1) \geq \sum_{i=1}^r N_f(T,mp_i) + \sum_{j=1}^s N_f(T \cdot Q_j,m)-(r+s-1)N_f(T,m).$$
We can rewrite the equation in the form
\begin{align*}
N_f(x^n-1,q^n-1) & \geq \sum_{i=1}^r \Big[N_f(T,mp_i) - \theta(p_i) N_f(T,m) \Big]\\
& + \sum_{j=1}^s \Big[ N_f(T \cdot Q_j,m) - \Theta(Q_j)N_f(T,m) \Big]  + \delta N_f(T,m).
\end{align*}

Now we need a good bound for $N_f(T,mp_i) - \theta(p_i) N_f(T,m)$. Since $\theta$ is a multiplicative function, we have 
$$
N_f(T,mp_i) = \Theta(T) \theta(m) \theta(p_i) \sum_{ \alpha \in \mathbb{F}_{q^n}}
\displaystyle 
\int_{d|mp_i} \int_{D|T} \eta_d(L_f(\alpha)) \chi_{\delta_D}(\alpha).
$$
We split the set of $d$'s which divide $mp_i$ into two sets: the first one contains those
which do not have $p_i$ as a factor, while the second one contains those which are a
multiple of $p_i$. This will split the first summation into two sums, and we get
\begin{align*}
N_f(T,mp_i) & = \Theta(T) \theta(m) \theta(p_i) \sum_{ \alpha \in \mathbb{F}_{q^n}} \left( \displaystyle
\int_{d|m} \int_{D|T}
\eta_d(L_f(\alpha))   \chi_{\delta_D}(\alpha) \right) \\
& + \Theta(T) \theta(m) \theta(p_i) \sum_{ \alpha \in \mathbb{F}_{q^n}} \left(
\displaystyle
\int_{\substack{d, \ p_i|d \\ d|mp_i}}\int_{D|T}
 \eta_d(L_f(\alpha)) \chi_{\delta_D}(\alpha) \right) .
\end{align*}
Hence, $N_f(T,mp_i) - \theta(p_i)N_f(T,m)$ is equal to
$$
\Theta(T) \theta(m) \theta(p_i) \sum_{ \alpha \in \mathbb{F}_{q^n}}
\left(
\displaystyle
\int_{\substack{d, \ p_i|d \\ d|mp_i}}\int_{D|T}
\eta_d(L_f(\alpha)) \chi_{\delta_D}(\alpha)
\right).
$$
So, from Lemma \ref{lema cota 2} we have the following inequality
$$
\Big| N_f(T,mp_i) - \theta(p_i)N_f(T,m) \Big| \leq \Theta(T) \theta(m) \theta(p_i) W(T)W(m) q^{n/2+k}.
$$
Analogously we can prove that
$$
\Big| N_f(T \cdot Q_j,m) - \Theta(Q_j)N_f(T,m) \Big| \leq \Theta(T) \theta(m) \Theta(Q_j) W(T)W(m) q^{n/2+k}.
$$
Combining these inequalities, we obtain
\begin{align*}
N_f(x^n-1,q^n-1) & \geq \delta N_f(T,m) - \\
&  \Theta(T) \theta(m) W(T)W(m)q^{n/2+k} \left( \sum_{i=1}^r \theta(p_i) + \sum_{j=1}^s \Theta(Q_j) \right).
\end{align*}
Therefore, 
from Theorem \ref{condicao em N-f}, we have
\begin{align*}
N_f(x^n-1,q^n-1) & > \delta \Theta(T) \theta(m) \left( q^n- q^{n/2+k}W(m)W(T) \right) \\ &  \ \  \ - \Theta(T) \theta(m) W(T)W(m)q^{n/2+k} \left( \sum_{i=1}^r \theta(p_i) + \sum_{j=1}^s \Theta(Q_j) \right) \\
& = \delta \Theta(T)\theta(m) \Big( q^n - q^{n/2+k}W(m)W(T) \Delta \Big).
\end{align*}
Thus, we obtain the desired result.
\end{proof}


For the case $k \geq 2$ we can rewrite the previous condition as follows, depending on the factorization of $x^n-1$.

\begin{proposition}\label{caseall} 
Let $n \geq 5$ be a natural number and let $q$ be a prime power such that $q\geq n^2$.
If $x^n-1$ has a factor of degree $k\geq 2$
in $\F_q[x]$ and $q^{\frac{n}{2}-k} \geq (n+2) W(q^n-1)$, then 
there exists a primitive $k$-normal element in $\F_{q^n}$.
\end{proposition}
\begin{proof}
Let $f \in \F_q[x]$ be a factor of $x^n-1$ of degree $k$.
We may use Proposition \ref{lema3.5} with $T=1$ and $m=q^n-1$.

Let $Q_1,\ldots , Q_s$ be irreducible polynomials such that
$rad(x^n-1)= Q_1 \cdot Q_2 \cdots Q_s$.
Then
$\delta=1-\sum_{j=1}^{s}\frac{1}{q^{\deg Q_j}} \geq 1 - \frac{n}{q}
\geq 1 - \frac{1}{n} = \frac{n-1}{n} >0$, since
$q \geq n^2$ and $s \leq n$.  We also have that
$$
\Delta=\frac{s-1}{\delta}+2\leq \frac{n-1}{\frac{n-1}{n}}+2 = n+2.
$$
This means that $W(m) W(T)  \Delta \leq (n+2)W(q^n-1)$ and from 
Proposition \ref{lema3.5}, we get
the desired result.
\end{proof}

For small values of $q$ we have the following result which will be used in combination with
Theorem \ref{condicao em N-f} and Lemma \ref{cota-t}. Note that those results
are different from \cite[Lemma 2.11]{lenstra}.
\begin{lemma}\label{Wxn-1}
	For $q$ a prime power, there exist $a,b \in \mathbb{N}$ such that
	$$
	W(x^n-1) \leq 2^{\frac{n+a}{b}}.
	$$
	For $q \geq 29$, we have $a=0$ and $b=1$, for
	$7 \leq q \leq 27$ we have
	$a=q-1$ and $b=2$ and for small values of $q$ we may use the following
	values of $a$ and $b$.
	\begin{table}[h]
		\centering
		\begin{tabular}{ccc}
			$q$ & $a$ & $b$ \\
			\hline
			$2$   & $14$ & $5$  \\
			$3$   & $20$ & $4$  \\
			$4$ & $12$  & $3$  \\
			$5$ & $18$  & $3$ 
		\end{tabular}
		\caption{Values of $a$ and $b$ for small values of $q$}
		\label{Wforsmallq}
	\end{table}
\end{lemma}
\begin{proof}
Let $s_{n,t}$ be the number of distinct monic irreducible polynomials of degree at most $t$
that divide $x^n-1$ and let $T_{n,t}$ be the sum of their degrees.
Hence $W(x^n-1)=2^j$, where
\begin{equation}\label{sum-u_i}
j \leq \frac{n-T_{n,t}}{t+1} + s_{n,t}
=\frac{n+ (t+1)s_{n,t} -T_{n,t}}{t+1}.
\end{equation}
Since each term in the sum $T_{n,t}$ is at most $t$, the right-hand side of the expression above
maximizes when $s_{n,t}$ is maximal.
On the other hand, it is obvious that zero is not a root of $x^n-1$,
so the sum of the degrees of polynomials of degree $i$ is less or equal than the number of elements of $\F_{q^i}^*$,
which is not an element of $\F_{q^j}^*$, for any divisor $j$ of $i$.

Table \ref{Wforsmallq} is obtained from 
\eqref{sum-u_i} and the reasoning above.
For $q=2$, we use $t=4$;
for $q=3$, we use $t=3$;
for $q=4$ or $q=5$, we use $t=2$; and for $7 \leq q \leq 27$, we use $t=1$ to obtain
$a=q-1$ and $b=2$. For $q \geq 29$, it is convenient to use the usual inequality.
%
\end{proof}

\section{Results for all cases where $n \geq 8$ and the cases $n \geq 5$ for $q \leq 19$}\label{sectionalmost}

In this section we begin to apply the results of the previous section for the case $k=2$. Thus,
we
study the values of $q$ and $n$ for which we can guarantee the existence of primitive $2$-normal elements in $\F_{q^n}$. 

\begin{proposition} \label{casesq19}
Let $q\leq 19$ be a prime power 
and $n \geq 5$ be a natural number. There exists a
primitive $2$-normal element in $\F_{q^n}$ if and only if
$\gcd (q^3-q,n) \neq 1$.
\end{proposition}
\begin{proof}
From Theorem \ref{condicao em N-f}, if
$q^{\frac{n}{2} -2} \geq W(q^n-1)W(x^n-1)$ then $N_f(x^n-1,q^n-1)>0$. From Lemma \ref{cota-t} and
Lemma \ref{Wxn-1}, it follows that
$A_t \cdot q^{\frac{n}{t}}  \cdot 2^{\frac{n+a}{b}} \geq W(q^n-1)W(x^n-1)$,
where
$$
A_t=\prod_{\substack{\wp < 2^t \\ \wp \text{ is prime}
		\\ \wp \neq p}}
\frac{2}{\sqrt[t]{\wp}} \quad \text{and} \quad 
\mathop{\text{char}} \F_q =p.
$$
So,
if for some $t \in \mathbb{N}$ we have 
$q^{\frac{n}{2} -2} \geq A_t \cdot q^{\frac{n}{t}}  \cdot 2^{\frac{n+a}{b}}$,
then
$N_f(x^n-1,q^n-1)>0$.
For $b>\log_q 4$ and $t> \dfrac{2b}{b - \log_q 4}$,
this inequality is equivalent to
\begin{equation}\label{conditionB}
n \geq 
\frac{2\ln q + \ln \left( A_t \cdot 2^{\frac{a}{b}}\right)}
{\left( \frac{1}{2} - \frac{1}{t} \right) \ln q - \frac{1}{b} \ln 2}.
\end{equation}

\begin{table}[h]
	\centering
	\begin{tabular}{cccc|cccc}
		$q$ & $a$ & $b$ & \eqref{conditionB} satisfied for &
		$q$ & $a$ & $b$ & \eqref{conditionB} satisfied for  \\
		\hline 
		$2$  &  $14$ & $5$  & $n \geq 69$ & 
		$8$ & $7$     & $2$ & $n \geq 28$ \\
		$3$  & $20$ & $4$  & $n \geq 46$ & 
		$9$ &  $8$      & $2$ & $n \geq 27$ \\
		$4$ & $12$  & $3$ & $n \geq 38$ &
		$11$ & $10$  & $2$  & $n \geq 26$ \\
		$5$ & $18$  & $3$  & $n \geq 35$ & 
		$16$ & $15$  & $2$  & $n \geq 24$ \\
		$7$ & $6$    & $2$  & $n \geq 31$ &
		$\{13,17,19\}$ & $q-1$  & $2$  & $n \geq 25$
\end{tabular}\vspace*{0.5cm}
	\caption{Values of $n$ depending on $q$ such that \eqref{conditionB} is satisfied with $t=6$.}
	\label{satisfiedA}
\end{table}


We know that for values of $q$ and $n$ from Table \ref{satisfiedA}, the
condition 
$q^{n/2 -2} \geq W(q^n-1)W(x^n-1)$ is satisfied, so
it remains only a finite number of cases to test for $q \leq 19$.

\begin{table}[h]
	\centering
	\begin{tabular}{|c|c||c|c|}
		\hline 
		$q$ & $n$ & $q$ & $n$ \\
		\hline 
		$2$  &  $6,8,9,10,12,14,15,18,21$ & $9$ & $5,6,8,10$ \\
		$3$  & $6,8,9,10,12,16$ &  $11$ & $5,6,8,10,12$ \\ 
		$4$ & $5,6,8,9,10,12,15$&  $13$ & $6,7,8,12$ \\
		$5$ & $5,6,8,9,10,12,16$ & $16$ & $5,6,9,10,15$ \\
		$7$ & $6,7,8,9,10,12$ & $17$ & $6,8$  \\
		$8$ & $6,7,8,9$ & $19$ & $5,6,8,9,10,12$ \\
		\hline
	\end{tabular}\vspace*{0.5cm}
	\caption{
		Values of $q$ and  $n$ such that $q \leq 19$,
		$n$ is not in Table \ref{satisfiedA},
		$\gcd(q(q-1)(q+1),n)\neq 1$ and
		$q^{n/2 -2} < W(q^n-1)W(x^n-1)$}
	\label{satisfiedB}
\end{table}


Table \ref{satisfiedB}  shows 
the values of $q$ and $n$ 
which are not in Table \ref{satisfiedA}
with $q \leq 19$ and $\gcd(q^3-q,n)\neq 1$,
where $q^{n/2 -2} \geq W(q^n-1)W(x^n-1)$ is not satisfied.

For pairs $(q,n)$ from Table \ref{satisfiedB}, 
we test condition $q^{\frac{n}{2}-2} \geq W(m) W(T)  \Delta$
(see Proposition \ref{lema3.5}). 
For this, we use the SageMath procedure
{\bf Test\_Delta(q,n,u)} (see Appendix A)
where $u$ is a given natural number,  $m=\gcd(q^n-1,u)$.
If $q>5$ we choose $T=1$;
if $q \leq 5$ then $T$ is the product of all monic linear factors of $x^n-1$.

For $u=2\cdot 3 \cdot 5$, procedure {\bf Test\_Delta(q,n,u)} gets True for
$(q,n) = (2, 14)$,
$(2, 15)$, $(2, 18)$,
$(2, 21)$, $(3, 9)$,
$(3, 16)$, $(4, 10)$,
$(4, 12)$, $(4, 15)$,
$(5, 9)$, $(5, 10)$,
$(5, 16)$, $(7, 7)$,
$(7, 9)$, $(7, 10)$,
$(8, 8)$, $(8, 9)$,
$(9, 10)$, $(11, 8)$,
$(11, 10)$, $(11, 12)$,
$(13, 7)$, $(13, 8)$,
$(16, 9)$, $(16, 10)$,
$(17, 8)$, $(19, 8)$,
$(19, 9)$, $(19, 10)$, $(19, 12)$.

For $(q,n)=(7,12)$, $(13,12)$, $(16,15)$,  we take 
$m=30,30,3$ and
$T= x^2-1$, $x^{4}-1$, $x^{15}-1$ respectively, and 
we get that condition $q^{\frac{n}{2}-2} \geq W(m) W(T)  \Delta$
is satisfied.

For the last remaining cases,
Tables \ref{Especificq23}, \ref{Especificq48} and \ref{Especificq919} (see Appendix B),
show explicitly a primitive $2$-normal element 
$\alpha \in \F_{q^n}$ such that $g(\alpha)=0$ for
some irreducible polynomial $g \in \F_p[x]$, where $p$ is the characteristic
of $\F_q$. Primitivity and Normality can be tested using the programs in Appendix A.
\end{proof}

\begin{proposition}\label{casesn19}
Let $n \geq 8$ be a natural number. There exists a
primitive $2$-normal element in $\F_{q^n}$ for all prime powers $q$
satisfying $\gcd (q^3-q,n) \neq 1$.
\end{proposition}
\begin{proof}
From the last result, 
we have that
for $q< 23$ and $n \geq 5$, there  exists a primitive $2$-normal element in $\F_{q^n}$.
So we will
focus on $q \geq 23$.
From Theorem \ref{condicao em N-f}, 
Lemma \ref{cota-t} and
considering that $W(x^n-1)\leq 2^n$,
there exists a primitive $2$-normal element in $\F_{q^n}$ if
$q^{\frac{n}{2}-2} \geq 2^n \cdot A_t \cdot q^{\frac{n}{t}}$
is
satisfied for some real number $t$.
This condition is equivalent to the following two inequalities:
$$
n  \geq  \frac{\ln(A_t) + 2 \ln(q)}{\left(\frac{1}{2}-\frac{1}{t} \right) \ln(q) - \ln(2)} 
\quad \text{and}
\quad 
q \geq
\left(
2^n \cdot A_t
\right)^{\frac{2t}{(t-2)n-4t}} .
$$
For a fixed value of $t \geq 4$,
the right-hand side of the first inequality is a decreasing function of $q\geq 16$. So,
fixing $t=7$ in the first inequality, we get that for
$q \geq 23$ and $n \geq 28$,
there exists a primitive $2$-normal element in $\F_{q^n}$. Now, if $14 \leq n \leq 27$,
from the second inequality (whose right-hand side is also a decreasing function of $n$) and
with $t=6.3$, we get that, 
if $n \geq 14$ and $q \geq 144$,
there exists a primitive $2$-normal element in $\F_{q^n}$. 
Now, using SageMath, we verify that
$q^{\frac{n}{2}-2} \geq W(q^n-1) W(x^n-1)$ is true for
all prime powers $23 \leq q <144$ and $14 \leq n < 28$. Hence, from Theorem \ref{condicao em N-f} and the previous considerations we conclude that
there exists a primitive $2$-normal element in $\F_{q^n}$
for every prime power $q$ and for all $n \geq 14$.

Now, let us suppose that $8 \leq n \leq 13$. 
From Proposition \ref{caseall} and
Lemma \ref{cota-t} there exists a primitive $2$-normal element in $\F_{q^n}$ if
$q \geq n^2$ and 
$q^{\frac{n}{2}-2} \geq (n+2) A_t \cdot q^{\frac{n}{t}}$. 
Define 
\begin{equation}\label{Mtn}
M_t(n) =
\max \left\{
n^2 ,
\left\lceil
\left(
(n+2) \cdot A_t
\right)^{\frac{2t}{(t-2)n-4t}} 
\right\rceil
\right\} ,
\end{equation}
where for $x \in \mathbb{R}$, $\lceil x \rceil$ is the smallest
integer such that $x \leq \lceil x \rceil$.
From the inequalities above, if 
we have $q \geq M_t(n)$, for some real number $t$
suficiently large (for $n \geq 8$ this means $t > 4$),
then there exists a primitive $2$-normal element in $\F_{q^n}$.
For $n$ between $8$ and $13$, we have
$$
\begin{array}{lll}
M_{6.3}(8)=6426, & M_6(10)=100,& M_6(12)= 144,\\
M_{6}(9) =413, & M_6(11)=121, & M_6(13)=169.
\end{array}
$$
For pairs $(q,n)$ such that $8 \leq n \leq 13$ and
$23 \leq  q < M_t(n)$, where $q$ is a prime power, 
we test $q^{\frac{n}{2}-2} \geq W(m) W(T)  \Delta$,
from Proposition \ref{lema3.5}.
The procedure {\bf Test\_Delta}, with $u=6$,
returns True for all those pairs.
Proposition \ref{casesn19} is now proved.
\end{proof}

\section{Cases $n=5,6,7$}\label{section=567}
For $5 \leq n \leq 7$, 
applying Proposition \ref{caseall} and Lemma \ref{cota-t}, we get that a sufficient condition to have
a primitive $2$-normal element in $\F_{q^n}$ is $q \geq M_t(n)$ for some real number $t$, where
$M_t(n)$ is defined by equation \eqref{Mtn}.
The problem is that $M_t(n)$ is very large.

\subsection{Case n=7:} The condition $\gcd (q^3-q,7) \neq 1$
means that $q\equiv 0,\pm 1 \pmod 7$.

\begin{proposition}\label{casesn7}
There exists a primitive $2$-normal element in $\F_{q^7}$
for every prime power $q$ such that $q\equiv 0,\pm 1 \pmod 7$. 
\end{proposition}
\begin{proof}
Suppose first that $7 \mid q$. In this case,
$q=7^k$ for some integer $k \geq 1$. We will use Theorem \ref{condicao em N-f}
in combination with  Lemma \ref{cota-t}. Since $7 \nmid q^n-1$,
we may use the following constant
\begin{equation}\label{cota-t7}
A_t=\prod_{\substack{\wp \neq 7 \, , \, \wp < 2^t  \\ \wp \text{ is prime}}}
\frac{2}{\sqrt[t]{\wp}}
\end{equation}
from Lemma \ref{cota-t}. From Theorem \ref{condicao em N-f},
and taking into account that $W(x^7-1)=W((x-1)^7)=2$,
we have that if, for some real number $t$, the inequality
$q^{\frac{7}{2}-2} \geq W(x^7-1) \cdot A_t \cdot q^{\frac{7}{t}} = 
2A_t \cdot q^{\frac{7}{t}}$ holds, then $N_f(x^7-1,q^7-1)>0$.
Setting $t=7$, we get that
$N_f(x^7-1,q^7-1)>0$ for $q \geq 104368$.
Since for prime powers $q=7^2,7^3,7^4,7^5$ the condition
$q^{\frac{7}{2} -2} \geq W(q^7-1)W(x^7-1)$
is satisfied, 
the result follows from Theorem \ref{condicao em N-f}.

If $q\equiv - 1 \pmod 7$, then $7 \nmid q^7-1$, and we may also use
the constant $A_t$ given by \eqref{cota-t7}.
From Lemma \ref{divisores de x^n-1}, we conclude that
$x^7-1$ has one factor of degree $1$ and three factors
of degree $2$.
We set
$m=q^7-1$ and $T=1$, so
$\delta = 1 - \frac{1}{q} - \frac{3}{q^2}$ and
$\Delta= \frac{4-1}{\delta}+2$.
Since $q \geq 23$ and $q \equiv -1 \pmod 7$ then, $q \geq 27$.
This 
means that $\Delta < 5.116$ and from
Proposition \ref{lema3.5}
we get $N_f(x^7-1,q^7-1)>0$ for prime powers $q$ satisfying
$$
q^{\frac{7}{2}-\frac{7}{t}-2} \geq 5.116 \cdot A_t > A_t  \cdot \Delta,
$$
for some real number $t$.
Setting $t=6.5$, we get $N_f(x^7-1,q^7-1)>0$ for
$q \geq 614236$. There are 
$8377$ prime powers $q$ between $23$ and $614236$
such that $q \equiv -1 \pmod 7$. For those prime powers, we
use
Theorem \ref{condicao em N-f} and we get that condition
$q^{\frac{7}{2}-2} \geq W(q^7-1)W(x^7-1)$ is satisfied except for
$q=27$. From $27^7-1=2 \cdot 13 \cdot 1093 \cdot 368089$,
we set  $m=27^7-1$ and $T=x-1$ in Proposition \ref{lema3.5}
and we get $27^{\frac{7}{2}-2} \geq W(m)W(T) \Delta$, so
the proposition is proved for
$q \equiv -1 \pmod 7$.

Finally, suppose that $q \equiv 1 \pmod 7$. In this case
we may use the following constant
$$
A_t=\frac{2}{\sqrt[t]{7^2}} \cdot
\prod_{\substack{\wp \neq 7 \, , \, \wp < 2^t  \\ \wp \text{ is prime}}}
\frac{2}{\sqrt[t]{p}}
$$
from Lemma \ref{cota-t}, as $7^2$ appears in the factorization
of $q^7-1$ and $7^2<2^t$ for any $t\geq 6$.
From Lemma \ref{divisores de x^n-1} we know that $x^7-1$ has seven factors of degree $1$.
We set
$m=q^7-1$ and $T=1$, so
$\delta = 1 - \frac{7}{q}$ and
$
\Delta= \frac{6}{\delta}+2 = 8 + \frac{42}{q-7}.
$
Let us suppose that $q\geq 337$. This means that
$\Delta < 8.128$
and from Proposition \ref{lema3.5}, we get that if
$q^{\frac{7}{2}-2} \geq 8.128 \cdot  A_t \cdot q^{\frac{7}{t}}>  W(q^7-1)W(1)\Delta$,
then $N_f(x^7-1,q^7-1)>0$.
Setting $t=6.8$, the inequality 
$q^{\frac{3}{2}-\frac{7}{t}} \geq 8.128\cdot  A_t$
is equivalent to
$q \geq 2142829$. For those prime powers, we have $N_f(x^7-1,q^7-1)>0$. 
There are 
$26543$ prime powers $q$ between $23$ and $2142829$
such that $q \equiv 1 \pmod 7$. 
For those prime powers,
we test $q^{\frac{n}{2}-2} \geq W(m) W(T)  \Delta$. The
procedure {\bf Test\_Delta}, with $u=2$,
returns True in all cases, so
the proposition is also proved for
$q \equiv 1 \pmod 7$.
\end{proof}

\subsection{Case n=6:} The condition $\gcd(q^3-q,6)\neq 1$ is satisfied for
every prime power $q$. 
From the considerations at the beginning
of this section, we have
$N_f(x^6-1,q^6-1)>0$ for prime powers $q \geq M_t(6)$. For $t=8.1$ we get
$M_t(6)< 1.62\cdot 10^{18}$. So, we will suppose that 
$q <1.62\cdot 10^{18}$.

We have that if $q$ is a prime power then $q$ is of the form
$2^k$, $3^k$ or $q \equiv \pm 1 \pmod 6$. 
Observe also that 
$\gcd(q^2-1,q^4+q^2+1)=\gcd(q^2-1,3)=3$.
Let
$$
q^4 +q^2 + 1 = 3^{\beta_0} \cdot \prod_{i=1}^v \wp_i^{\beta_i}
$$
be the prime factorization of $q^4 +q^2 + 1$,
where $3 < \wp_1 < \cdots < \wp_v$ are the prime factors of $q^4+q^2+1$. 

Now, we want to apply Proposition \ref{lema3.5}
with $m=q^2-1$, and therefore we need to have some control on the prime factors of $q^4 +q^2 + 1$.


\begin{lemma} \label{soma6}
Let $q \equiv \pm 1 \pmod 6$ be a prime power such that
$q <1.62\cdot 10^{18}$. If
$q^4 +q^2 + 1 = 3^{\beta_0} \cdot \prod_{i=1}^r {\wp}_i^{\beta_i}$
is the prime factorization of $q^4+q^2+1$,
then $r \leq 34$ and
$$
S=\sum_{i=1}^r \frac{1}{{\wp}_i} < 0.539 .
$$
\end{lemma}
\begin{proof}
Let $S_k$ and $P_k$ be, respectively,
the sum of the inverses and the product of the first $k$ primes of the form
$6j+1$. 
We have that $\gcd (q^4+q^2+1,q^2-1)=1$ or $3$,
so the only primes which divide
$q^4 +q^2 + 1$ are $3$ and primes of the form $6j+1$. Thus
$q^4 +q^2 + 1 \geq 3 \cdot P_r$ and then
$P_r \leq (M^4+M^2+1)/3$,
where $M=1.62\cdot 10^{18}$. We have that $P_r \leq (M^4+M^2+1)/3$ for $r \leq 34$, so
$S \leq S_{34} < 0.539$.
\end{proof}

\begin{proposition}
There exists a primitive $2$-normal element in $\F_{q^6}$
for every prime power $q$.
\end{proposition}
\begin{proof}
Consider $10^5< q < 1.62\cdot 10^{18}$ and let us suppose first that $q \equiv \pm 1 \pmod 6$. 
Now, we will apply Proposition \ref{lema3.5}
with $m=q^2-1$ and $T=1$.
If $q \equiv 1 \pmod 6$, then $x^6-1$ has six factors of degree $1$ and
if $q \equiv -1 \pmod 6$, then $x^6-1$ has two factors of degree $1$ and
two factors of degree $2$. In any case, we have 
$\frac{6}{q} < \frac{2}{q}+\frac{2}{q^2}$ and $s=4$ or $s=6$, so, in any case, $s \leq 6$.
From Lemma \ref{soma6},
considering the prime factorization of $q^4+q^2+1$ given in such lemma
and the considerations above,
 we get
$\delta \geq 1 - S_{34}- \frac{6}{q}$, $r\leq 34$ and $s \leq 6$. Since
$q \geq 10^5$, we have $\delta >0.46094$ and
$\Delta = 2 + \frac{r+s-1}{\delta} < 86.61$.
From Lemma \ref{cota-t} we have
$W(q^2-1) \leq A_t \cdot q^{\frac{2}{t}}$ for any real number $t$.
Now, 
if $q \geq \left( A_t \cdot 86.61 \right)^{\frac{t}{t-2}}$,
then
from Proposition \ref{lema3.5},
there exists a primitive $2$-normal element in $\F_{q^6}$.
For $t=4.9$, this condition becomes $q \geq 94870$.
Now, let us assume that $q<94870$ and $q \equiv \pm 1 \pmod 6$. 
There are $9221$ prime powers $q$ between $23$ and $94870$
such that $q \equiv  \pm 1 \pmod 6$.
For those prime powers,
we test $q^{\frac{n}{2}-2} \geq W(m) W(T)  \Delta$. The
procedure {\bf Test\_Delta}, with $u=6$,
returns True in all cases, except for the following
prime powers:
$23$, $25$, $29$,  $31$, $37$, $41$, $43$, $47$, $49$, $59$, $61$, $67$, $79$.

Finally, for prime powers above,
table \ref{n=6_1} shows an element
$\alpha \in \F_{q^6}$, primitive $2$-normal over $\F_q$,
such that $g(\alpha)=0$ for
some irreducible polynomial $g \in \F_p[x]$, where $p$ is the characteristic
of $\F_q$.

If $q=2^k$, then $W(x^6-1) \leq 8$ and if $q=3^k$, then $W(x^6-1) = 4$.
Hence from Theorem \ref{condicao em N-f} and Lemma \ref{cota-t}, we test the 
inequality $q\geq 8 \cdot A_t \cdot q^{\frac{6}{t}}$ with $t=8$ and we conclude
that there exists a primitive $2$-normal element in $\F_{q^6}$ for 
$k \geq 58$ (if $q=2^k$) and $k \geq 37$ (if $q=3^k$). We also know
that there exists a primitive $2$-normal element in $\F_{q^6}$ for $q \leq 19$.
We test the condition $q \geq W(m) W(T)  \Delta$ from Proposition \ref{lema3.5}.
The procedure {\bf Test\_Delta}, with $u=3$
for $q=2^k$ ($5 \leq k \leq 57$) and $u=2$ for $q=3^k$ ($3 \leq k \leq 36$),
returns True in all these cases.
\end{proof}

\subsection{Case n=5:} 
From Lemma \ref{divisores de x^n-1}, if $q \equiv \pm 2 \pmod 5$, then
$x^5-1$ has no irreducible quadratic factor and only one linear factor.
If $5 \, | \, q$, we have $x^5-1=(x-1)^5$,
if $q \equiv 1 \pmod 5$, then $x^5-1$ has five linear factors
and if 
$q \equiv -1 \pmod 5$, then $x^5-1$ has one linear factor
and two irreducible factors of degree $2$.
In particular, there exist $2$-normal
elements in $\F_{q^5}$ if and only if $q\equiv 0,\pm 1 \pmod 5$.

\begin{lemma}\label{casen5}
Let $q\equiv 0,\pm 1 \pmod 5$ be a prime power.
There exists a primitive $2$-normal element in $\F_{q^5}$
for $q \geq 507936$.
\end{lemma}
\begin{proof}
Let $t,u$ be positive real numbers such that $t+u \geq 11$ and
let 
$$
q^5-1=\wp_1^{a_1} \cdots \wp_v^{a_{v}} \cdot
\varrho_1^{b_1} \cdots \varrho_r^{b_{r}}
$$
be the prime factorization of $q^5-1$ such that $2 \leq \wp_i \leq 2^t$  or 
$2^{t+u} \leq \wp_i$
for $1 \leq i \leq v$ and
$2^t < \varrho_i < 2^{t+u}$
for $1 \leq i \leq r$.
We use Proposition \ref{lema3.5}, where
we set $T=1$ and
$m= \wp_1^{a_1} \cdots \wp_v^{a_{v}}$, so we have
$$
\delta = 1 - \sum_{i=1}^r \frac{1}{{\varrho}_i} - \sum_{j=1}^s \frac{1}{q^{\deg Q_j}},
$$
where $\mathop{\text{rad}} (x^n-1) = Q_1 \cdots Q_s$. From the considerations above, we have $s \in \{ 1,3, 5\}$ and $1 \leq \deg Q_i \leq 2$. 
Before applying Proposition \ref{lema3.5}, we will bound $\delta$, $\Delta$ and $W(m)$.

Let $S_{t,u}$ be the 
sum of the inverse of all prime numbers between $2^t$ and $2^{t+u}$ and
$r(t,u)$ be the number of those primes.
If $S_{t,u} + \frac{5}{q} <1$, then
$\delta \geq 1 - S_{t,u} - \frac{5}{q}$.
If we choose $q >10^6$ and $(t,u)=(5.8,9.8)$, we get
$S_{t,u} \leq 0.962094$, $\delta > 0.037901$, $r\leq r(t,u)=5085$
and $\Delta =2 + \frac{r+s-1}{\delta}<2+\frac{5085+5-1}{0.037901}\leq
134272.87$.
To bound $W(m)$ we will use Lemma \ref{cota-t}.
Let  $P_t$ be the set of all prime numbers less than $2^t$.
From this, we obtain that $W(m) \leq A_{t,u} m^{\frac{1}{t+u}} \leq A_{t,u} q^{\frac{5}{t+u}}$,
where
$$
A_{t,u}=\prod_{\wp\in P_t}
\frac{2}{\sqrt[t+u]{\wp}} \leq 3678.26, \quad \text{since } (t,u)=(5.8,9.8).
$$
From Proposition \ref{lema3.5}, we conclude that a suficient condition for the existence of 
a primitive $2$-normal element in $\F_{q^5}$ is
$q^{\frac{1}{2}} \geq \Delta_{max} \cdot A_{t,u} \cdot q^{\frac{5}{t+u}}$
(where $\Delta_{max}=134272.87$)
or, equivalently,
$$
q \geq  \left( 
\Delta_{max} \cdot A_{t,u}
\right)^{ \frac{2(t+u)}{t+u-10}} \cong 2.729 \cdot 10^{48}.
$$

Let's suppose now that $q < 2.729 \cdot 10^{48}$. We will apply Proposition \ref{lema3.5} again, but this time we will set
$m=q-1$ and $T=1$. 
We have that $\gcd (q^4+q^3+q^2+q+1,q-1)=1$ or $5$ and if a prime 
different from $5$ divides $q^4+q^3+q^2+q+1$, then it is of the form
$5j+1$.
We will bound $\delta$, $\Delta$ and $W(m)$. Obviously, from 
Lemma \ref{cota-t}, we have 
$W(q-1) \leq A_t \cdot q^{\frac{1}{t}}$ for any real number $t$.
Let $S_k$ and $P_k$ be, respectively, the sum of the inverses
and the product of the first $k$ primes of the form $5j+1$.
Let $r$ be the number of prime factors of 
$q^4+q^3+q^2+q+1$
different from $5$, so $P_r \leq q^4+q^3+q^2+q+1 <5.55\cdot 10^{193}$.
Therefore $r \leq 69$ and $S_r < 0.29717$. As before, if $q>10^5$ then
$\delta \geq 1 - S_r - \frac{5}{10^5} >0.70278$ and
$\Delta = 2 + \frac{r+s-1}{\delta} <105.874$.
So, observing that if $q \geq \left( 105.874 \cdot A_t \right)^{\frac{2t}{t-2}}$ 
for some real number $t$, then 
$q^{\frac{1}{2}} \geq  A_t \cdot q^{\frac{1}{t}}  \cdot 105.874 \geq 
 W(q-1) \cdot  W(1) \cdot \Delta$, and using 
 Proposition \ref{lema3.5}, there exists a primitive $2$-normal element
in $\F_{q^5}$

 For $t=4.7$, the condition above becomes
 $q \geq 1.984\cdot 10^{10}$. If we suppose now $q<1.984\cdot 10^{10}$
 and if we use again Proposition \ref{lema3.5} with $m=q-1$ and $T=1$,
 we get $r \leq 19$ and $S_r <0.2441801$. For $q >10^5$, we also get
 $\delta > 0.7558149$ and $\Delta < 32.43074$.
From Proposition \ref{lema3.5} and taking $t=4.8$, we get that
there exists a primitive $2$-normal element in $\F_{q^5}$ for
$q \geq 3.208 \cdot 10^8$.

We  apply now
Proposition \ref{lema3.5}, 
setting $m=\gcd (q^5-1 , 2 \cdot 3 \cdot 5)$ (so $W(m) \leq 8$) and $T=1$,
and let $r$ and $s\leq 5$ be the natural numbers defined by Proposition \ref{lema3.5}. 
Let
$S_k$ and $P_k$ be, respectively,
the sum of the inverses and the product of the first $k$ primes greater than $5$.
In particular we have $P_r \leq M^5-1 \leq 3.398 \cdot 10^{42}$,
where $M=3.208 \cdot 10^8$. This implies that
$r \leq 25$, and if we suppose $q \geq 10^6$ then
$\delta \geq 1 - S_r - \frac{5}{q} > 0.20155$ and $\Delta \leq 145.885$. The condition
from Proposition \ref{lema3.5} is
$q^{\frac{1}{2}} \geq 1167.08 \geq 8  \cdot  145.885  \geq W(m) W(T)  \Delta$.
So, if 
$q \geq 1.363 \cdot 10^6$, then there exists a primitive $2$-normal element 
in $\F_{q^5}$. 

Finally, we apply one last time 
Proposition \ref{lema3.5}, 
setting $m=\gcd (q^5-1 , 2 \cdot 3 \cdot 5)$ (so $W(m) \leq 8$) and $T=1$.
This time we get $r \leq 19$ and $S_r \leq 0.7359$. If we suppose $q \geq 10^6$,
we get $\delta \geq 0.264104$, $\Delta \leq 89.087$
and $q \geq 507936$.
\end{proof}

If we try to use procedure {\bf Test\_Delta}, with $u=2\cdot 3$,
for all prime powers such that 
$q \equiv 0,\pm 1 \pmod 5$ and $q <507936$, 
it will produce a list of $127$ prime powers for which {\bf Test\_Delta} returns False. For this reason
we will try another approach.

\begin{lemma}\label{factorx5-1}
Let $q$ be a prime power such that $q \equiv \pm 1 \pmod 5$. Then
$x^5-1=(x-1)(x^2-bx+1)(x^2+(b+1)x+1)$, where
$b\in \F_q$ is a root of $x^2+x-1=0$.
\end{lemma}
\begin{proof}
Let $\xi \neq 1$ be a root of $x^5-1$ in $\F_{q^2}$ and define
$b=\xi + \xi^{-1}$. If $q \equiv 1 \pmod 5$, then obviously
$\xi\in\F_q$,
which implies that $b \in \F_q$. If $q \equiv -1 \pmod 5$, then
 there exists a primitive element $\alpha$ in $\F_{q^2}$ such that
 $\xi =\alpha^{\frac{q^2-1}{5}}$. Observe that
 $$
 \xi^q = \alpha^{\frac{q(q-1)(q+1)}{5}} = 
\left( \alpha^{\frac{q+1}{5}} \right)^{q^2-q}=
\left( \alpha^{\frac{q+1}{5}} \right)^{1-q}=
\xi^{-1}.
 $$
This implies that $b^q=\xi^{-1}+\xi=b$, so we also  have $b \in \F_q$.
Since $\xi^4+\xi^3+\xi^2+\xi+1=0$, we get that
$
b^2+b=
\xi^{-2} (\xi^4+\xi^3+\xi^2+\xi+1) + 1 = 1
$
and $(x^2-bx+1)(x^2+(b+1)x+1)=x^4+x^3+x^2+x+1$.
\end{proof}

\begin{lemma}\label{factorx5-1normal2}
Let $q$ be a prime power such that $q \equiv \pm 1 \pmod 5$,
$b\in \F_q$ be a root of $x^2+x-1=0$, $\alpha$ be a normal
element in $\F_{q^5}$ and $f=x^2-bx+1$. Then
$L_f(\alpha)+a$ is a $2$-normal element in
$\F_{q^5}$ for all $a \in \F_q$ except for only one value of $a$.
\end{lemma}
\begin{proof}
If we let $g= (x-1)(x^2+(b+1)x+1)$, we get $fg=x^5-1$ and, for every element $\gamma \in \F_{q^5}$
we have $0=L_{fg}(\gamma)=L_g(L_f(\gamma))$.
Since $x-1$ is a factor of $g$, then $L_g(a)=0$ for every
$a \in \F_q$. In particular, if 
$\alpha$ is a normal element in $\F_{q^5}$ then
$L_g(L_f(\alpha)+a)=L_g(L_f(\alpha))+L_g(a)=0$ for
every $a \in \F_q$,
so $L_f(\alpha)+a$ has $F_q$-order $\deg h\leq 3$ for some divisor $h$ of $g$. 
From Theorem \ref{equiv-knormal},  we get that $L_f(\alpha)+a$ is
$k$-normal where $k\geq 2$. Let us suppose that $\deg h \leq 2$. 
If $x-1 \mid h$, then $L_h(a)=0$ and $L_h(L_f(\alpha))\neq 0$, since $\alpha$ is normal,
so, in this case, $L_h(L_f(\alpha)+a)\neq 0$. This means that if 
$\deg h \leq 2$, then $x-1 \nmid h$ and, in particular, $h \mid x^2+(b+1)x+1$.
Note that $L_{x^2+(b+1)x+1} (L_f(\alpha)+a)=0$ is equivalent to
$L_{x^4+x^3+x^2+x+1}(\alpha)+L_{x^2+(b+1)x+1} (a)=0$.
Since $L_{x^2+(b+1)x+1} (a)=(b+3)a$, then 
$\alpha^4+\alpha^3+\alpha^2+\alpha+1=-(b+3)a$. 
If $b=-3$, then $(-3)^2+(-3)-1=0$, which is not possible because $5 \nmid q$ and hence
$b\neq -3$. Therefore there is only one possible value of $a$ such that
$Tr_{q^5/q}(\alpha)=-(b+3)a$. In particular, this means that if
$a \neq -(b+3)^{-1} \cdot Tr_{q^5/q}(\alpha)$, then $L_f(\alpha)+a$ is
$2$-normal.
\end{proof}

In fact, if $j \in \F_q$ is the only value for which $L_f(\alpha)+j$
is not $2$-normal, then
$g(a)=(a-j)(L_f(\alpha)+a)$ is $2$-normal for every $a \in \F_q \backslash \{ j \}$
and for $j$ we have $g(j)=0$. This means that if $g(a)$ is primitive, then $g(a)$ is 
also $2$-normal.
We finish the case $n=5$ with a computational approach
using the idea from Lemma \ref{factorx5-1normal2}.
\begin{proposition}\label{finaln5}
Let $q$ be a prime power. 
There exists a primitive element $2$-normal in $\F_{q^5}$ if and only if
$q \equiv 0,\pm 1 \pmod 5$. 
\end{proposition}
\begin{proof}
From Lemma \ref{casen5},
we only need to prove the existence of a 
primitive $2$-normal element in $\F_{q^5}$ for
$q \equiv 0,\pm 1 \pmod 5$ such that $q <507936$.

 
Inspired by Lemma \ref{factorx5-1normal2}, we use the
SageMath
procedure
named
{\bf TestExplicit5} (see Appendix A)
to find
a primitive $2$-normal element in $\F_{q^5}$. In this procedure, $a$ generates $\F_{q^5}$,
$j \in \F_p$,
$b \in \F_q$ is a root of $x^2+x-1=0$ and $\beta = L_{x^2-bx+1}(a)$ for
$q \equiv \pm 1 \pmod 5$. If $5 \, | \, q$, then
$\beta =L_{(x-1)^2}(a)$. From Lemma \ref{factorx5-1normal2}, 
we get that if $q\equiv \pm 1 \pmod 5$, then
$\beta + j$ is always $2$-normal except, maybe, for one value of $j$. In any case
($q\equiv 0 \pmod 5$ or $q \equiv \pm 1 \pmod 5$), this procedure
returns True if $\beta + j$ is primitive $2$-normal in $\F_{q^5}$ for some $j \in \F_p$,
where $p = \mathop{\text{char}} \F_q$.

For all prime powers $q$ such that
$q \equiv 0,\pm 1 \pmod 5$ for which $q <507936$ and
{\bf Test\_Delta} (with $u=2\cdot 3$) returns False,
the procedure {\bf TestExplicit5(q)} returns False only for $q=64$.

For $q=64$, we may use procedures 
{\bf ordmodqn} and {\bf Normal} (see Appendix A)
to see that $\alpha$ is a primitive $2$-normal element
in $\F_{q^5}$ where $\alpha$ is a root of
$$
g(x) = 
x^{30} + x^{27} + x^{26} + x^{23} + x^{22} 
+ x^{21} + x^{16} +
 x^{14} + x^{12} + x^9  
+ x^6 + x^5 + x^3 + x + 1.
$$
This proves the proposition.
\end{proof}

\section{Case $n=4$}\label{section=4}

In \cite{lucas} after Remark 3.5, the author proved
that there is no 
primitive $2$-normal element in $\F_{q^4}$ if $q \equiv 3 \pmod 4$. Suppose now that $q$ is a power of $2$. In this case, $x^4-1=(x+1)^4$ and 
$f=(x+1)^2$, so 
if $\beta$ is a $2$-normal element, there exists a normal element $\alpha \in \F_{q^4}$
such that $\beta = L_{(x+1)^2}(\alpha)=\alpha^{q^2}+\alpha$.
Since $\beta^{q^2}=(\alpha^{q^2}+\alpha)^{q^2}=\alpha+\alpha^{q^2}=\beta$,
we have that $\beta$ is not a primitive element in $\F_{q^4}$. Therefore, if there exists a primitive $2$-normal element in $\F_{q^4}$, then $q \equiv 1 \pmod 4$. In this case, we may factor $x^4-1$ 
into four linear factors, say $x^4-1=(x+1)(x-1)(x+b)(x-b)$, where $b \in \F_q$ and $b^2=-1$.

Throughout this section, we will consider
$q \equiv 1 \pmod 4$, $b\in \F_q$ such that $b^2=-1$,
$f(x)=(x+1)(x+b) \in \F_{q}[x]$ a factor of $x^4-1$ of degree two and
$\alpha \in \F_{q^4}$ a 
normal 
element in $\F_{q^4}$.
Thus,  $L_f(\alpha)$ is a $2$-normal element in $\F_{q^4}$ (see \cite{lucas}, Lemma 3.1). The following result tells us that we can generate more $2$-normal elements if they are also primitive, more precisely we have

\begin{lemma}\label{primes2normal}
Let $u,v \in \F_q^*$ and $f(x)=(x+1)(x+b) \in \F_{q}[x]$, where $b \in \F_q$ satisfies $b^2=-1$. If $\gamma=u L_f(\alpha)+v$ is primitive in $\F_{q^4}$, then $\gamma$ is $2$-normal in $\F_{q^4}$.
\end{lemma}

\begin{proof}
We know that $L_f$ is a linear transformation over $\F_q$, so $L_{(x^4-1)/f}(u L_f(\alpha)+v)= uL_{(x-1)(x-b)}(L_f(\alpha))+L_{(x-1)(x-b)}(v)= uL_{x^4-1}(\alpha)+L_{(x-b)}\left(v^q-v \right)=0$. Since
$\frac{x^4-1}{f}$ is a degree two polynomial, we have that
 the set $\{\gamma, \gamma^q, \gamma^{q^2}\}$ is linearly dependent. Suppose that $\gamma$ and $\gamma^q$ are linearly dependent, thus $\gamma^{q-1} \in \F_q^*$ and $\ord_q(\gamma) \leq (q-1)^2<q^4-1$, which is a contradiction. Thus, $\langle \gamma,\gamma^q, \gamma^{q^2}, \gamma^{q^3}\rangle=\langle \gamma, \gamma^q \rangle$ and $\gamma$ is $2$-normal, by Theorem \ref{equiv-knormal}.
\end{proof}

Let us define a function $g(x)=x+ \beta $,
where $\beta$ is a $2$-normal element in $\F_{q^n}$.
We need conditions for the existence of primitive 
elements
of the form $g(a)$, where $a \in \F_q^*$, because
in the case where $n=4$, 
if we choose $\beta =L_f(\alpha)$ then, from Lemma \ref{primes2normal},
primitivity of $g(a)$ implies $2$-normality of $g(a)$. Observe also that if $\beta =L_f(\alpha) \in \F_{q^4}$, then
$\beta \notin \F_{q^2}$. Indeed, if $\beta \in \F_{q^2}$,
then $0=L_{\frac{x^4-1}{f}}(\beta)=\beta^{q^2}-(b+1)\beta^q +b\beta = (b+1)(\beta - \beta^q)$. Since $b^2=-1$ and $q \equiv 1 \pmod 4$,
we get $b \neq -1$ and
$\beta \in \mathbb{F}_q$, which is a contradiction.




\begin{theorem}\label{condn4}
Let $q$ be a prime power, 
let $m\in \mathbb{N}$ be a divisor of $q^4-1$,
and let $\beta=L_f(\alpha)$ be a $2$-normal element in 
$\F_{q^4}$, where 
$\alpha \in \F_{q^4}$ is a normal element,
$f(x)=(x+1)(x+b) \in \F_{q}[x]$ and $b \in \F_q$ satisfies $b^2=-1$.
Let $N_\beta(m)$ be the number of  elements $a \in \F_q$ such that $g(a)=a+\beta$ is $m$-free. If $q^{1/2} \geq 3W(m)$,
then $N_\beta(m)>0$, i.e., there exists an element of the form $g(a)$ in $\F_{q^4}^*$ which is $m$-free.
\end{theorem}
\begin{proof}
Since $\beta=L_f(\alpha) \notin \mathbb{F}_{q^2}$, we have that $\mathbb{F}_{q^4}=\mathbb{F}_q(\beta)$. 
Therefore, from
Lemma \ref{cotaenFq}, for any non-trivial multiplicative character $\chi$ over $\F_{q^4}$,
we have
\begin{equation}\label{cotaparag}
 \left| \sum \limits_{a \in \F_q} \chi(g(a)) \right| \leq 3 \sqrt{q},
\end{equation}
Now, from Proposition \ref{caracteristica-mlivre}, we have that
\begin{align*}
N_\beta(m) & =\sum_{a \in \F_q} w_m\left( g(a) \right) = 
\theta(m) \left( \sum_{ a \in \F_q} \chi_1(g(a)) + \int \limits_{d|m, \ d \neq 1} \sum_{ a \in \F_q} \chi_d(g(a)) \right).
\end{align*}
Therefore we obtain the following estimative, using inequality \eqref{cotaparag}
\begin{align*}
\left| \frac{N_\beta(m)}{\theta(m)} - q \right| & \leq \left |  \sum_{\substack{d|m \\ d \neq 1}}  
\dfrac{\mu(d)}{\varphi(d)} \sum_{(d)} \sum_{a \in \F_q} \chi_d(g(a)) \right| \leq 3 \sqrt{q} \sum_{\substack{d|m \\ d \neq 1}} |\mu(d)|  
= 3(W(m)-1)\sqrt{q},
\end{align*}
Then $\frac{N_\beta(m)}{\theta(m)}\geq  q-3(W(m)-1)\sqrt{q}$, and we obtain the desired result.
\end{proof} 

The next result's proof is similar to the proof of Proposition \ref{lema3.5}, and hence is omitted.


\begin{proposition}\label{sieven4}
Let $m \in \mathbb{N}$ be a divisor of $q^4-1$
and let $\beta$ be a $2$-normal element as in Theorem \ref{condn4}.
Let ${\wp}_1,\ldots , {\wp}_r$ be prime numbers such that
$rad(q^4-1)=rad(m) \cdot {\wp}_1 \cdot {\wp}_2 \cdots {\wp}_r$.
Suppose that 
$\delta=1-\sum_{i=1}^{r}\frac{1}{{\wp}_i}>0$
and let  $\Delta=\frac{r-1}{\delta}+2$. If
$q^{\frac{1}{2}} \geq 3  W(m) \Delta$,
then $N_\beta(q^4-1)>0$.
\end{proposition}

Now,  we get the following sufficient conditions for the existence of primitive elements in
$\F_{q^4}$ of the form $g(a)=a + \beta$. If
\begin{equation}\label{cond-g(a)-primitive}
q^{1/2}> 3 W(q^4-1) \quad \text{or}
\quad 
q^{1/2}> 3 W(m) \Delta
\end{equation}
(for some $m \mid q^4-1$ and a specific value of $\Delta$),
then there exists a primitive element in $\F_{q^4}$ of the form $g(a)=a+\beta$
with $a \in \F_q$; also this element is $2$-normal by Lemma \ref{primes2normal}.
Let us use Proposition \ref{sieven4} in combination with Lemma
\ref{primes2normal} to find a bound for the values of $q$
such that there exists a primitive $2$-normal element in $\F_{q^4}$.

\begin{theorem}\label{teo-cota-1}
Let $q$ be a prime power. There exists a primitive $2$-normal element
in $\F_{q^4}$ if and only if $q \equiv 1 \pmod 4$.
\end{theorem}
\begin{proof}
We proceed as in Lemma \ref{casen5}.
Let $t,u$ be positive real numbers such that $t+u > 8$ and
let $q^4-1=\wp_1^{a_1} \cdots \wp_v^{a_{v}} \cdot
\varrho_1^{b_1} \cdots \varrho_r^{b_{r}}$
be the prime factorization of $q^4-1$ such that $2 \leq \wp_i \leq 2^t$  or 
$2^{t+u} \leq \wp_i$
for $1 \leq i \leq v$ and
$2^t < \varrho_i < 2^{t+u}$
for $1 \leq i \leq r$ and	
consider $m= \wp_1^{a_1} \cdots \wp_v^{a_{v}}$.
Let $S_{t,u}<1$ be the 
sum of the inverses of all prime numbers between $2^t$ and $2^{t+u}$, and
$r(t,u)$ be the number of those primes. As in Lemma \ref{casen5}
$r \leq r(t,u)$, $\delta  \geq 1-S_{t,u}$ and
$\Delta \leq 2+ \frac{r(t,u)-1}{1 - S_{t,u}}$.
By Lemma \ref{cota-t},
considering that $2^4 \mid q^4-1$ and $4 < t+u$,
we have $W(m) < A_{t,u} \cdot q^{\frac{4}{t+u}}$, where
$$
A_{t,u} = 
\frac{2}{\sqrt[t+u]{2^4}} \cdot 
\prod \limits_{\substack{2<\wp < 2^t \\ \wp \text{ is prime}}} \frac{2}{\sqrt[t+u]{\wp}}.
$$
We know that $\beta =L_f(\alpha) \notin \F_{q^2}$
and we may apply Proposition \ref{sieven4}.
Therefore, if
$q^{\frac{1}{2}}\geq 3\cdot A_{t,u} \cdot q^{\frac{4}{t+u}} \cdot \Delta$
then $N_\beta(q^4-1)>0$. This condition is equivalent to
$q \geq \left(3\cdot A_{t,u} \cdot \Delta \right)^{\frac{2(t+u)}{t+u-8}}$. Taking $t=5$ and $u=8.5$
we get 
$N_\beta(q^4-1)>0$ for $q\geq M=2.12\cdot 10^{35}$.

Suppose now $q < M=2.12\cdot 10^{35}$.
We will use now
Proposition \ref{sieven4} with $m=q^2-1$.
Let $q^2+1=2 \cdot \wp_1^{a_1} \cdots \wp_r^{a_{r}}$
be the prime factorization of $q^2+1$.
For any odd prime number such that
$\wp \mid q^2+1$, we have 
$q^2 \not\equiv 1 \pmod{\wp}$ and
$q^4 \equiv 1 \pmod{\wp}$.
This means that $4 \mid \varphi(\wp)=\wp-1$.
Let $S_k$ be the sum of the inverses of the first $k$ prime numbers of the form $4j+1$ and
let
$P_k$ be the product of those $k$ prime numbers. So, from
$2P_r \leq q^2+1 < M^2+1$, we get $r \leq 33$, $S_r <0.60520004$,
$\delta >0.39479996$ and $\Delta \leq 83.054$.
Let
$$
A_{t} = 
\frac{2}{\sqrt[t]{2^3}} \cdot 
\prod \limits_{\substack{2<\wp < 2^t \\ \wp \text{ is prime}}} \frac{2}{\sqrt[t]{p}}
$$
be the constant from Lemma \ref{cota-t},
considering that $2^3 \mid q^2-1$ and $3 < t$.
Therefore, if 
$q^{\frac{1}{2}} \geq 3 \cdot A_t \cdot q^{\frac{2}{t}} \cdot  \Delta> 3 \cdot W(q^2-1) \cdot  \Delta$, and applying
Proposition \ref{sieven4}, then
$N_\beta(q^4-1)>0$. For $t=6.8$, we get 
$\left( 3 \cdot A_t \cdot \Delta \right)^{\frac{2t}{t-4}} \leq 7.321\cdot 10^{21}$.

Let us suppose now that  $M=7.321\cdot 10^{21}$ and $q < M$. We will use again
Proposition \ref{sieven4} with
$m=\gcd(q^4-1,2\cdot 3 \cdot 5 \cdot 7)$. 
Let $S_k$ be the sum of the inverses of the first $k$ prime numbers starting with
$11$ and
let $P_k$ be the product of those $k$ prime numbers.
Observe that
if $5 \nmid q^4-1$, then $q$ is a prime power of $5$ which implies that $3 \mid q^4-1$.
This means that $2^4 \cdot 3 \mid q^4-1$ or $2^4 \cdot 5 \mid q^4-1$. Let  $r$ be the
number of prime factors of $q^4-1$ greater than $7$. 
We have $r \leq 44$, since $P_r  < \frac{M^4-1}{48}$. So,
$S \leq S_r < 0.7821$, $\delta >0.2179$ and
therefore $\Delta< 2+\frac{44-1}{0.2179}<199.34$.
Thus,
$q^{\frac{1}{2}} \geq 3 \cdot W(m) \cdot \Delta$
for $q \geq 9.156\cdot 10^7$.

We repeat this last process
with $M=9.156\cdot 10^7$
and $m=\gcd(q^4-1,2\cdot 3 \cdot 5)$. Now 
$S_k$ is the sum of the inverses
of the first $k$ prime numbers starting with
$7$ and $P_k$ is the product of those $k$ prime numbers.
We have $P_r < \frac{M^4-1}{48}$ for $r \leq 19$
and $\Delta < 70.155$.
So, 
$q^{\frac{1}{2}} \geq 3 \cdot 2^3 \cdot 70.155 \geq 3 \cdot W(m) \cdot \Delta$
for $q \geq 2834914$. Repeating
this process one last time with $M=2834914$ and 
$m=\gcd(q^4-1,2\cdot 3 \cdot 5)$, we get $r \leq 16$,
$\Delta <51.253$. From Proposition \ref{sieven4} we get
$N_\beta(q^4-1)>0$ if $q \geq 1513078 >
\left(3 \cdot 2^3 \cdot  51.253\right)^2$.
From Lemma \ref{primes2normal},
we get that, for $q\geq 1513078$, there exists a primitive $2$-normal element
in $\F_{q^4}$.

There are $57731$ prime powers $q \equiv 1 \pmod 4$
less than $1513078$. We use the test $q^\frac{1}{2} \geq 3 \cdot W(m) \cdot \Delta$ using 
the SageMath procedure
{\bf Test(q,list)} from Appendix A where the variable $list$ is the list of prime numbers
which can be factors of $m$. With $list=[2,3,5]$, the procedure {\bf Test(q,list)}
returns False
for $1704$ primes powers from all prime powers $q \equiv 1 \pmod 4$
less than $1513078$. For those prime powers,
the procedure {\bf Test(q,list)}, with
$list=[2,3]$, returns False for $934$ prime powers. Finally,
for those last prime powers, the procedure {\bf Test(q,list)}, with
$list=[2]$, returns False for $918$ prime powers.

Now,
we use the SageMath procedure named
{\bf TestExplicit4} (see Appendix A)
to find
a primitive $2$-normal element in $\F_{q^4}$. In this procedure, 
we found first a normal element $\alpha \in \F_{q^4}$, 
$b \in \F_q$ a root of $x^2+1=0$
and we define $\beta =L_{(x+1)(x+b)}(a)$. Next, we try to find an element
$j \in \F_p$  such that $\beta + j$ is primitive.
From Lemma \ref{primes2normal}, 
$\beta + j$ is also $2$-normal. This procedure
returns True if $\beta + j$ is primitive $2$-normal in $\F_{q^4}$ for some $j \in \F_p$.
For all $918$ prime powers for which we didn't conclude
with the procedure {\bf Test(q,list)}, the procedure
{\bf TestExplicit4(q)} returns
False only for $13,17,125$.
Table \ref{n=4} shows for
these cases a primitive $2$-normal element
$\alpha \in \F_{q^4}$,
such that $h(\alpha)=0$, for
some irreducible polynomial $h \in \F_p[x]$, where $p$ is the characteristic
of $\F_q$. This completes the proof.
\end{proof}

%
%
%

\section*{Acknowledgements}
Victor G.L.\ Neumann was partially funded by FAPEMIG APQ-03518-18. The authors
would like to thank 
C\'icero Carvalho, Carol Lafetá and the referees
for very useful comments and suggestions that improved the presentation of this work.



\section*{Appendix A: Procedures in SageMath}

{\small
\begin{verbatim}
def Test_Delta(q,n,u):
   A.<a>=GF(q); P.<x>=PolynomialRing(A)
   M1=factor(q^n-1); M2=factor(x^n-1) ; m=1
   count1=0; count2=0; choose=True
   T=1
   for g in M2:
      if g[0].degree()==1 and q<7:
         T=T*g[0]
   r=len(M1); s=len(M2); S1=0; S2=0
   for p in M1:
      if gcd(p[0],u)!=1:
         m=m*p[0]
         r=r-1
      else:
         S1=S1+1/p[0]
   for Q in M2:
      if gcd(Q[0],T)!=1:
         s=s-1
      else:
         S2=S2+1/q^(Q[0].degree())
   delta=1-S1-S2
   if delta>0:
      Delta=2+(r+s-1)/delta; A=(q*1.0)^(n*0.5-2)
      B=Delta*2^(len(factor(m)))*2^len(factor(T)); Fact=A>=B
   else:
      Fact=False
   return Fact
\end{verbatim}
}
---------------------------------------------------------------------------------------------------------
{\small
\begin{verbatim}
#Given p,q,n and g define:
B=GF(p); T.<x>=PolynomialRing(B)
C.<c>=B.extension(g); R.<x>=PolynomialRing(C)
#Testing if c is primitive: ordmodqn(c)
#Testing if c is 2-normal: Normal(c)
#where:
def ordmodqn(d):
   ord=q^n-1
   for m in divisors(q^n-1):
      if d^m==1:
         if m<ord:
            ord=m
   return (q^n-1)/ord
def Normal(e):
   pol=0
   for i in range(0,n):
      pol=pol+e^(q^i)*x^(n-1-i)
   pol_gcd=gcd(pol,x^n-1); k=pol_gcd.degree()
   return k
\end{verbatim}
}
---------------------------------------------------------------------------------------------------------
{\small
\begin{verbatim}
def TestExplicit5(q):
   A.<a>=GF(q^5); T.<x>=PolynomialRing(A)
   if mod(q,5)==0:
      beta=a^(q^2)-2*a^q+a
   else:
      Sol=(x^2+x-1).roots(); b=Sol[0][0]; beta=a^(q^2)-b*a^q+a
   j=0; Teste=False; valor=True
   while valor:
      c=beta+j; ord=q^5-1
      for m in divisors(q^5-1):
         if c^m==1:
            if m<ord:
               ord=m
      e=(q^n-1)/ord
      if e==1:
         pol=0
         for i in range(0,5):
            pol=pol+c^(q^i)*x^(4-i)
         pol_gcd=gcd(pol,x^n-1); k=pol_gcd.degree()
         if k==2:
            valor=False; Teste=True
      j=j+1
      if beta+j==beta:
         valor=False
   return Teste
\end{verbatim}
}
---------------------------------------------------------------------------------------------------------
{\small
\begin{verbatim}
def Test(q,list):
   L=factor(q^4-1); m=1; r=len(L); S=0.0
   for p in L:
      if p[0] in list:
         m=m*p[0]; r=r-1
      else:
         S=S+1/p[0]
   delta=1-S
   if delta>0:
      Delta=2+(r-1)/delta; A=(q*1.0)^(0.5); B=3*Delta*2^(len(factor(m)))
      Fact=A>=B
   else:
      Fact=False
   return Fact
\end{verbatim}
}
---------------------------------------------------------------------------------------------------------
{\small
\begin{verbatim}
def TestExplicit4(q):
   A.<a>=GF(q^4, modulus="primitive"); T.<x>=PolynomialRing(A)
   z=1
   norm=True
   while norm:
      alpha=a^z
      pol=0
      for i in range(0,n):
         pol=pol+alpha^(q^i)*x^(n-1-i)
      pol_gcd=gcd(pol,x^n-1); k=pol_gcd.degree()
      if k==0:
         norm=False
      z=z+1
      if z==q^4-1 and norm:
         norm=False; Test=False 
   Sol=(x^2+1).roots(); b=Sol[0][0]
   beta=alpha^(q^2)+(b+1)*alpha^q+b*alpha; j=0
   Test=False; valor=True
   while valor:
      c=beta+j; ord=q^n-1
      for m in divisors(q^n-1):
         if c^m==1:
            if m<ord:
               ord=m
      e=(q^n-1)/ord
      if e==1:
         pol=0
         for i in range(0,n):
            pol=pol+c^(q^i)*x^(n-1-i)
         pol_gcd=gcd(pol,x^n-1); k=pol_gcd.degree()
         if k==2:
            valor=False; Test=True
      j=j+1
      if beta+j==beta:
         valor=False
   return Test
\end{verbatim}
}

\section*{Appendix B: Tables}
\begin{table}[H]
	\centering
	\begin{tabular}{cc}
		$(q,n)$ & $g(x) \in \F_p[x]$ \\
		\hline
		$(2,6)$   &
		$x^6 + x^5 + x^3 + x^2 + 1$  \\ \hline
		$(2,8)$   &
		$x^8 + x^5 + x^3 + x + 1$  \\ \hline
		$(2,9)$ & 
		$x^9 + x^8 + x^6 + x^5 + x^4 + x^3 + x^2 + x + 1$ 
		\\ \hline
		$(2,10)$ & 
		$x^{10} + x^6 + x^5 + x^3 + x^2 + x + 1$ 
		\\ \hline
		$(2,12)$ &
		$x^{12} + x^{10} + x^8 + x^4 + x^3 + x^2 + 1$
		\\ \hline	
		$(3,6)$ & 
        $x^6 + x^5 + x^4 + x^3 + x + 2$
        \\ \hline
		$(3,8)$ & 
        $x^8 + 2x^5 + x^4 + 2x^2 + 2x + 2$
\\ \hline
		$(3,10)$ & 
        $x^{10} + x^8 + x^7 + 2x^6 + x^5 + x^4 + x^3 + 2x^2 + x + 2$
\\ \hline
		$(3,12)$ & 
        $x^{12} + x^{10} + 2x^9 + 2x^8 + x^7 + x^6 + 2x^4 + 2x^3 + 2$
\\ \hline
	\end{tabular}\vspace*{0.5cm}
	\caption{$\alpha \in \F_{q^n}$ is a primitive
		$2$-normal element such that $g(\alpha)=0$}
	\label{Especificq23}
\end{table}

\begin{table}[H]
	\centering
	\begin{tabular}{cc}
		$(q,n)$ & $g(x) \in \F_p[x]$ \\
		\hline
		$(4,5)$   & $x^{10} + x^8 + x^6 + x^5 + x^3 + x + 1$		
		\\ \hline
		$(4,6)$   & $x^{12} + x^{11} + x^{10} + x^8 + x^6 + x^4 + x^3 + x + 1$		
		\\ \hline
		$(4,8)$ & $x^{16} + x^{13} + x^{12} + x^{11} + x^{10} + x^9 + x^8 + x^7 + x^5 + x^2 + 1$
		\\ \hline
		$(4,9)$ & $x^{18} + x^{16} + x^{12} + x^{10} + x^4 + x + 1$
		\\ \hline
		$(5,5)$ & $x^5 + 2x^3 + x + 2$
		\\ \hline
		$(5,6)$ & $x^6 + x^4 + 4x^3 + x^2 + 2$
		\\ \hline
		$(5,8)$ & $x^8 + 4x^7 + x^6 + 3x^4 + x^3 + x + 3$
		\\ \hline
		$(5,12)$ & $x^{12} + x^{11} + 3x^{10} + x^9 + 4x^7 + 3x^5 + 3x^3 + 3x^2 + 4x + 3$
		\\ \hline
		$(7,6)$ & $x^6 + x^4 + 5x^3 + 4x^2 + 6x + 3$
		\\ \hline
		$(7,8)$ & $x^8 + 3x^6 + 6x^5 + x^4 + 6x^3 + 5x^2 + 4x + 5$
        \\ \hline
        $(8,6)$ & $x^{18} + x^{16} + x^{15} + x^{14} + x^{13} + x^6 + x^2 + x + 1$
        \\ \hline
        $(8,7)$ & $x^{21} + x^{16} + x^{14} + x^{11} + x^7 + x^6 + x^5 + x^3 + 1$
        \\ \hline
\end{tabular}\vspace*{0.5cm}
\caption{$\alpha \in \F_{q^n}$ is a primitive $2$-normal element such that $g(\alpha)=0$}
\label{Especificq48}
\end{table}

\begin{table}[H]
	\centering
	\begin{tabular}{cc}
		$(q,n)$ & $g(x) \in \F_p[x]$ \\
		\hline
        $(9,5)$ & $x^{10} + x^7 + 2x^6 + x^5 + x^4 + 2x^3 + 2$
        \\ \hline 
        $(9,6)$ & $x^{12} + x^9 + x^8 + x^7 + x^6 + 2x^4 + 2x^3 + 2x + 2$
        \\ \hline
        $(9,8)$ & $x^{16} + 2x^{14} + 2x^{13} + 2x^{12} + x^{11} + x^{10} + 2x^9 + x^8 + x^5 + x^4 + 2$
        \\ \hline
        $(11,5)$ & $x^5 + 9x^3 + 4x^2 + 9x + 3$
        \\ \hline
        $(11,6)$ & $x^6 + 9x^5 + x^4 + 3x^3 + x^2 + x + 7$
		\\ \hline
		$(13,6)$ & $x^6 + 10x^3 + 11x^2 + 11x + 2$
		\\ \hline
		$(16,5)$ & $x^{20} + x^{19} + x^{15} + x^{13} + x^{11} + x^{10} + x^7 + x^6 + x^3 + x + 1$
		\\ \hline
		$(16,6)$ &
		$\begin{matrix}
		x^{24} + x^{22} + x^{21} + x^{20} + x^{19} + x^{18} + x^{15} + \\
		x^{14} + x^{12} + x^{10} + x^8 + x^7 + x^3 + x^2 + 1
		\end{matrix}
		$
		\\ \hline
		$(17,6)$ & $x^6 + 9x^5 + 15x^4 + 6x^3 + x^2 + 4x + 14$
		\\ \hline
		$(19,5)$ & $x^5 + 2x^4 + x^2 + 2x + 16$
		\\ \hline
		$(19,6)$ & $x^6 + 17x^3 + 17x^2 + 6x + 2$
		\\ \hline
\end{tabular}\vspace*{0.5cm}
\caption{$\alpha \in \F_{q^n}$ is a primitive $2$-normal element such that $g(\alpha)=0$}
\label{Especificq919}
\end{table}

\begin{table}[H]
	\centering
	\begin{tabular}{cc}
		$q$ & $g(x) \in \F_p[x]$ \\
		\hline
		$23$   & $x^6 + 3x^5 + 20x^4 + 12x^3 + 6x + 11$  \\ \hline
		$25$   & $x^{12}+x^{11}+3x^{10}+x^9 + 4x^7 + 3x^5 + 3x^3 + 3x^2 + 4x + 3$  \\ \hline
		$29$ & $x^6 + 14x^4 + 22x^3 + 6x^2 + 2x + 15$ \\ \hline
		$31$ & $x^6+19x^3 + 16x^2 + 8x + 3$  \\ \hline
		$37$ & $x^6+35x^3 + 4x^2 + 30x + 2$ \\ \hline
		$41$ & $x^6 + 17x^4 + 19x^3 + 9x^2 + 38x + 17$ \\ \hline
		$43$ & $x^6+19x^3 + 28x^2 + 21x + 3$ \\ \hline
		$47$ & $x^6 + 35x^4 + 36x^3 + 36x^2 + 19x + 31$ \\ \hline
		$49$ & $x^{12} + 6x^{10} + 5x^9 + 6x^8 + 6x^7 
		+ 3x^6 + x^5 + 4x^3 + x^2 + 5x + 3$ \\ \hline
		$59$   & $x^6 + 13x^4 + 56x^3 + 15x^2 + 2x + 11$  \\ \hline
$61$   & $x^6+49x^3 + 3x^2 + 29x + 2$  \\ \hline
$67$ & $x^6 + 32x^5 + 58x^4 + 46x^3 + 22x^2 + 59x + 61$ \\ \hline
$79$ & $x^6 + 19x^3 + 28x^2 + 68x + 3$ \\ \hline
	\end{tabular} \vspace*{0.5cm}
	\caption{$\alpha \in \F_{q^6}$ is a primitive
		$2$-normal element such that $g(\alpha)=0$}
	\label{n=6_1}
\end{table}

\begin{table}[H]
	\centering
	\begin{tabular}{cc}
		$q$ & $h(x) \in \F_p[x]$ \\
		\hline
		$13$ & $x^4 + 11x^3 + 8x^2 + 6x + 11$ \\ \hline
		$17$ & $x^4 + 10x^2 + 5x + 3$ \\ \hline
		$125$ & $x^{12} + x^{10} + 3x^9 + 4x^8 + 3x^6 + 2x^5 + 2x^4 + 3x^3 + x^2 + 4x + 3$ \\ \hline
	\end{tabular} \vspace*{0.5cm}
	\caption{$\alpha \in \F_{q^4}$ is a primitive
		$2$-normal element such that $h(\alpha)=0$}
	\label{n=4}
\end{table}

\vspace*{2cm}

\noindent Josimar J. R. Aguirre \\
Departamento de Matem\'{a}tica \\
Universidade Federal de Uberl\^{a}ndia - UFU\\
Uberl\^{a}ndia - MG, Brazil - 38400-902 \\
josimar.mat@ufu.br
\vspace*{1cm}\\
Victor G.L. Neumann\\
Departamento de Matem\'{a}tica \\
Universidade Federal de Uberl\^{a}ndia - UFU\\
Uberl\^{a}ndia - MG, Brazil - 38400-902 \\
victor.neumann@ufu.br
\end{document}